\documentclass{amsart}

\usepackage{amssymb,amsfonts}
\usepackage[all,arc]{xy}
\usepackage{enumerate}
\usepackage{mathrsfs}
\usepackage{graphicx}

\newtheorem{thm}{Theorem}[section]
\newtheorem{cor}[thm]{Corollary}
\newtheorem{prop}[thm]{Proposition}
\newtheorem{lem}[thm]{Lemma}

\theoremstyle{definition}
\newtheorem{defn}[thm]{Definition}

\newtheorem*{ack*}{Acknowledgements:}

\theoremstyle{remark}

\makeatletter
\let\c@equation\c@thm
\makeatother
\numberwithin{equation}{section}

\usepackage[margin=1in]{geometry}

\newcommand{\deriv}[2]{\frac{\partial #1}{\partial #2}}
\newcommand{\tanv}[1]{\frac{\partial}{\partial #1}}

\DeclareMathOperator{\Hom}{Hom}
\DeclareMathOperator{\Sym}{Sym}
\DeclareMathOperator{\tr}{tr}
\DeclareMathOperator{\Id}{Id}
\DeclareMathOperator{\rk}{rk}
\DeclareMathOperator{\ad}{ad}

\DeclareMathOperator{\Vol}{Vol}
\DeclareMathOperator{\End}{End}
\DeclareMathOperator{\GL}{GL}

\DeclareMathOperator{\Ric}{Ric}
\DeclareMathOperator{\Scal}{Scal}
\DeclareMathOperator{\Aut}{Aut}

\title{Higher Rank Bergman Kernels on Compact Riemann Surfaces}

\begin{document}

\begin{abstract}
    Let X be a compact Riemann surface equipped with a real-analytic K\"ahler form $\omega$ and let E be a holomorphic vector bundle over $X$ equipped with a real-analytic Hermitian metric $h$. Suppose that the curvature of $h$ is Griffiths-positive. We prove the existence of a global asymptotic expansion in powers of $k$ of the Bergman kernel associated to $(\Sym^k E, \Sym^k h)$ and $\omega$. 
\end{abstract}

\author{Shin Kim}
\address{Department of Mathematics, Statistics, and Computer Science, University of Illinois Chicago, 
Chicago, IL 60607}
\email{skim651@uic.edu}

\maketitle

\tableofcontents

\section{Introduction}

Let $(X,\omega)$ be a compact $n$-dimensional K\"ahler manifold and let $L$ be a holomorphic line bundle over $X$ with a Hermitian metric $h$ such that the curvature form of $h$ is $-2\pi \sqrt{-1} \omega$. For each $k \in \mathbb{N}$,
\begin{equation*}
( \cdot, \cdot ) = \int_X \langle \cdot, \cdot \rangle_{h^k} \frac{\omega^n}{n!}
\end{equation*}
is an $L^2$ inner product on the finite dimensional vector space $H^0(X,L^k)$. Let $d_k = \dim (H^0(X,L^k))$ and let $\{s_i\}_{i=1}^{d_k}$ be an orthonormal basis of $H^0(X,L^k)$. The Bergman kernel $K_k$ is the section of $L^k \boxtimes \overline{L}^k$ over $X \times X$ defined by
\begin{equation*}
    K_k(y,x) = \sum_{i=1}^{d_k} s_i(y) \otimes \overline{s_i}(x)
\end{equation*}
and the Bergman function is the smooth function
\begin{equation*}
    B_k(x) = \sum_{i=1}^{d_k} |s_i(x)|_{h^k}^2.
\end{equation*}

The works of Tian \cite{Tian}, Zelditch \cite{Z}, Catlin \cite{Cat}, Ruan \cite{R}, Berman, Berndtsson, and Sj\"ostrand \cite{BBS}, Dai, Liu, and Ma \cite{DLM}, and Liu and Lu \cite{LL} show that the Bergman function has an asymptotic expansion. More precisely, for fixed nonnegative integers $N$ and $p$, there exist smooth functions $b_{1}, \dotsc, b_{N}$ such that 
\begin{equation} \label{eq: Bergman Expansion}
B_k(x) = k^n + b_{1}(x) k^{n-1} + \dotsb + b_{N}(x) k^{n-N} + O(k^{n-N-1}) 
\end{equation}
where the error term $O(k^{n-N-1})$ is bounded with respect to the $C^p$-norm. Moreover, $b_1(x) = \Scal_\omega(x)$ where $\Scal_\omega(x)$ is the scalar curvature of $\omega$. Because of this, the Bergman function is a central tool in the study of K\"ahler metrics of constant scalar curvature (\cite{Don, Yau}).

A generalization of the asymptotic expansion \eqref{eq: Bergman Expansion} is also known. Let $G$ be a holomorphic vector bundle equipped with a Hermitian metric. For each $k \in \mathbb{N}$, we consider the vector bundle $G \otimes L^k$. Then, the Bergman function $B_k$ is a smooth section of $\End(G \otimes L^k) \cong \End(G)$ and, for any fixed nonnegative integers $N$ and $p$, there exist smooth sections $b_{1}, \dotsc, b_{N}$ of $\End(G)$ such that \eqref{eq: Bergman Expansion} holds (\cite{BBS, Cat, DLM, LL}).

Now, let $X$ be a compact Riemann surface and let $\omega$ be a real analytic K\"ahler form on $X$. Let $E$ be a holomorphic vector bundle over $X$ with a real analytic Hermitian metric $h$ and suppose that the curvature of $h$ is Griffiths-positive. For each $k \in \mathbb{N}$, the volume form $\omega$ and the Hermitian metric $\Sym^k h$ induces an $L^2$ inner product on $H^0(X,\Sym^k E)$. The Bergman function $B_k$ is the section of $\End(\Sym^k E)$ defined by
$$
B_k(x) = \sum_{i=1}^{d_k} \langle \cdot, s_i(x) \rangle_{\Sym^k h} s_i(x),
$$
where $\{s_1, \dotsc, s_{d_k}\}$ is an orthonormal basis of $H^0(X, \Sym^kE)$. The main result of this paper is the following theorem.

\begin{thm} \label{thm: Main}
Let $N$ and $p$ be fixed nonnegative integers. There exist smooth sections $b_{k,0}$, $\dotsc$, $b_{k,N}$ of $\End(\Sym^k E)$ such that the $C^p$ norms of $b_{k,i}$ are $O(k^p)$ for all $i = 0, \dotsc, N$ and
$$
B_k(x) = b_{k,0}(x) k + \dotsb + b_{k,N}(x) k^{1-N} + O(k^{p-N})
$$
with respect to the $C^p$ norm. Moreover, $b_{k,0}$ and $b_{k,1}$ are universal quantities that depend on the K\"ahler form $\omega$, the curvature $F_{\Sym^k h}$, and their derivatives. In fact,
\begin{align*}
b_{k,0}(x) = \hspace{2pt} & \frac{\sqrt{-1}}{k} \Lambda F_{\Sym^k h}(x) \hspace{10pt} \text{ and } \\
    b_{k,1}(x) = \hspace{2pt} & -\frac{1}{2} \Lambda \Delta F_{\Sym^k h}(x) (\Lambda F_{\Sym^k h})^{-1}(x) + \frac{1}{2} \Scal_\omega \Id_{\Sym^k E} (x)\\
    & + \frac{\sqrt{-1}}{2} \left ( \Lambda (\nabla^{1,0} \Lambda F_{\Sym^k h}) (\Lambda F_{\Sym^k h})^{-1}  \wedge (\nabla^{0,1} \Lambda F_{\Sym^k h}) (\Lambda F_{\Sym^k h})^{-1} 
 \right) (x).
\end{align*}
\end{thm}

If $E$ is a line bundle, the theorem above recovers the known asymptotic expansion of the Bergman function for powers of positive line bundles on curves (see Corollary \ref{cor: line bundle computation}). Additionally, as in the case for line bundles, the Bergman function can be used to recover the Riemann-Roch formula, up to an error term, for vector bundles over curves that admit a Griffiths-positive Hermitian metric (see Corollary \ref{cor: Riemann-Roch}).

If $E = L_1 \oplus L_2$ is a direct sum of two positive line bundles, then $\Sym^k E \cong \bigoplus_{a+b = k} L_1^a \otimes L_2^b$ and Theorem \ref{thm: Main} yields the following corollary (see Corollary \ref{cor: direct sum}).
\begin{cor}
Let $L_1$ and $L_2$ be line bundles with positive real analytic Hermitian metrics $h_1$ and $h_2$. For any $a,b \geq 0$, let $B_{a,b}$ denote the Bergman function associated to $(L_1^a \otimes L_2^b, h_1^a \otimes h_2^b)$ and $\omega$. Let $N$ and $p$ be any fixed nonnegative integer. Then, there exist smooth functions $b_{a,b,0}, \dotsc, b_{a,b,N}$ such that
\begin{equation*}
    B_{a,b}(x) = b_{a,b,0}(x)(a+b) + \dotsb + b_{a,b,N}(x) (a+b)^{1-N} + O \left ((a+b)^{-N} \right )
\end{equation*}
with respect to the $C^p$ norm. Furthermore,
\begin{equation*}
    b_{a,b,0}(x) = \frac{a}{a+b} \Lambda_\omega \omega_1 + \frac{b}{a+b} \Lambda_\omega \omega_2
\end{equation*}
and
\begin{equation*}
    b_{a,b,1}(x) = \Scal_\omega(x) -\frac{1}{2}  \left ( a \Lambda_\omega \omega_1 + b \Lambda_\omega \omega_2 \right ) \Scal_{a\omega_1 + b\omega_2}(x). 
\end{equation*}
where $\omega_j$ is $\sqrt{-1}$ times the curvature of $h_j$ for each $j=1,2$.
\end{cor}

Finally, if $h$ is a Hermitian-Einstein metric on $E$, then Theorem \ref{thm: Main} yields the following corollary (see Corollary \ref{cor: HE}).
\begin{cor}
Suppose that $h$ is a Hermitian-Einstein metric. Let $N$ and $p$ be fixed nonnegative integers. Then, there exist smooth functions $b_{1}(x), \dotsc, b_N(x)$, that do not depend on $k$, such that
\begin{equation*}
    B_k(x,x) = \left ( b_{0}(x) \Id_{\Sym^k E}\right )k + \dotsb + \left ( b_{N}(x) \Id_{\Sym^k E} \right ) k^{1-N} + O(k^{-N})
\end{equation*}
with respect to the $C^p$ norm. In particular,
\begin{equation*}
b_{0}(x) = c
\end{equation*}
where $c\Id_E = \sqrt{-1} \Lambda F_h$ and 
\begin{equation*}
b_{1}(x) = \frac{1}{2}\Scal_\omega(x).
\end{equation*}
\end{cor}

To prove Theorem \ref{thm: Main}, we will generalize the phase and negligible amplitudes from \cite{BBS} to construct local reproducing kernels for vector bundles using power series methods. Then, we will use H\"ormander estimates to show that the local reproducing kernels glue together to give the Bergman kernel. The proof of the formulas for $b_{k,0}(x)$ and $b_{k,1}(x)$ will be given in Appendix B.

\begin{ack*}
I would like to thank my advisor Julius Ross for many useful discussions. I would also like to thank Julien Keller for his significant contributions to an early version of this work. Finally, I would like to thank Ruadha\'i Dervan, Julien Keller, Nicholas McCleerey, and Annamaria Ortu for helpful conversations. This work received partial support from DMS 1749447.
\end{ack*}

\section{The Diastatic Function}

Let $(X,\omega)$ be a compact Riemann surface and let $(E,h)$ be a holomorphic and Hermitian vector bundle over $X$. Suppose that $\omega$ and $h$ are real analytic and that $h$ is Griffiths-positive. Fix a point $x_0 \in X$. Let $e^{-\phi}$ be a local representation of the Hermitian metric $h$ with respect to a holomorphic normal frame of $E$ centered at $x_0$ on a coordinate neighborhood $U$. After replacing $U$ with a sufficiently small coordinate unit disk centered at $x_0$, we can write
$$
\phi(x) = \sum_{\alpha,\beta > 0} \frac{1}{\alpha! \beta!} \frac{\partial^{\alpha + \beta} \phi}{\partial x^\alpha \partial \overline{x}^\beta} (0) x^\alpha \overline{x}^\beta.
$$
We define $\psi : U \times U \to M_{r \times r}(\mathbb{C})$ by
\begin{equation} \label{eq : psi}
\psi(y,z) =  \sum_{\alpha,\beta > 0} \frac{1}{\alpha!\beta!} 
 \frac{\partial^{\alpha + \beta} \phi}{\partial x^\alpha \partial \overline{x}^\beta} (0) y^\alpha z^\beta.
\end{equation}
In the case when $(E,h)$ is a line bundle, it is not hard to show that the so called \emph{diastatic function} 
\begin{equation} \label{eq: diastatic}
-\phi(x) + \psi(y,\overline{x}) - \phi(y) + \psi(x,\overline{y}) 
\end{equation}
is independent of the chosen frame and coordinates (\cite{C}, Chapter 2, Proposition 1). In \cite{BBS}, the estimate
\begin{equation} \label{eq: estimate}
-\phi(x) + \psi(y,\overline{x}) - \phi(y) + \psi(x,\overline{y})  \leq -\delta|x-y|^2 
\end{equation}
was used in a crucial way to show that the Bergman kernel for a positive Hermitian line bundle admits an asymptotic expansion. In this section, we will generalize \eqref{eq: estimate} to the case when $(E,h)$ is a Griffiths-positive vector bundle.

Let $p$ be a positive integer. For each $i \in \{1, \dotsc, p\}$, let $X_i$ be a variable representing $r \times r$ matrices. We write $X_i = (X_{i,jk})_{1\leq j,k \leq r}$, where $X_{i,jk}$ denotes the coordinate function corresponding to the $j,k$-th entry of $X_i$. Define $Z^{(p)}(X_1, \dotsc, X_p) \in \mathbb{C} [[X_{1,11}, \dotsc, X_{p,rr}]]$ by 
\begin{equation}
    Z^{(p)}(X_1, \dotsc, X_p) = \log \left ( e^{X_1} \dotsb e^{X_p} \right ).
\end{equation}
If $X_i$ is close to the origin, then $e^{X_i}$ is close the to the identity matrix. Since compositions of analytic functions are analytic, there exists an open neighborhood $V \subseteq M_{r \times r} (\mathbb{C})$ of the origin such that the series $Z^{(p)}$ converges absolutely on $V^{\times p}$. In other words, $Z^{(p)} : V^{\times p} \to M_{r \times r}(\mathbb{C})$ is an analytic function.

We record some facts about real analytic matrix-valued functions.

\begin{defn}
Let $U$ be an open subset of $\mathbb{C}^n$ and let $\phi : U \to M_{r \times r}(\mathbb{C})$ be a matrix-valued function. If $x \in U$ and $M,\rho$ are positive real numbers, we say that $\phi \in C_{op,M,\rho}(x)$ if each entry of $\phi$ is smooth in a neighborhood of $x$ and
\begin{equation}
\left \| \partial^\alpha \overline{\partial}^\beta \phi (x) \right \|_{op} \leq M |\alpha + \beta|! \rho^{-|\alpha + \beta|} \label{eq:coeffs}
\end{equation}
for all $\alpha,\beta \in (\mathbb{Z}_{\geq 0})^n$.
\end{defn}

Let $U$ be an open subset of $\mathbb{C}^n$ and let $\phi : U \to M_{r\times r}(\mathbb{C})$ be a smooth matrix-valued function defined on $U$. The proofs of the real variable versions of the following facts can be found in Chapter 3 of \cite{J}.

\begin{itemize}
    \item $\phi$ is real analytic if and only if for every compact subset $S \subseteq U$, we can find $M$ and $\rho$ such that $\phi \in C_{op,M,\rho}(x)$ for all $x \in S$.  
    \item Let $x \in U$ and $\rho > 0$. Suppose that $\phi(y) = \sum_{\alpha,\beta \geq 0} c_{\alpha\beta} (y-x)^\alpha (\overline{y} - \overline{x})^\beta$ for all $y$ in the polydisc of radius $\rho$ centered at $x$ and $\sum_{\alpha,\beta \geq 0} \| c_{\alpha\beta} \|_{op} \rho^{|\alpha + \beta|} < \infty$. Then, $\phi$ is real analytic on the polydisc of radius $\rho$ centered at $x$ and $c_{\alpha\beta} = \frac{1}{\alpha! \beta!} \partial^\alpha \overline{\partial}^\beta \phi(x)$ for all $\alpha,\beta \in (\mathbb{Z}_{\geq 0})^n$.
    \item Suppose that $\phi$ is real analytic and let $x \in U$ be given. Then, there exists $\rho > 0$ small enough so that, for any $y$ and $z$ in the polydisc of radius $\rho$ centered at $x$, we can write \begin{equation} \label{eq:series expansion}
    \phi(z) = \sum_{\alpha,\beta \geq 0} \frac{1}{\alpha!\beta!} \partial^\alpha \overline{\partial}^\beta \phi(y) (z-y)^\alpha (\overline{z}-\overline{y})^\beta.
    \end{equation}
\end{itemize}.

For vector bundles of higher rank, we will generalize the function \eqref{eq: diastatic} as follows.  
\begin{defn}
We call the matrix valued function $D(x,y)$ such that
$$
e^{-\phi(x)/2} e^{\psi(y,\overline{x})} e^{-\phi(y)} e^{\psi(x,\overline{y})} e^{-\phi(x)/2} = e^{D(x,y)}
$$
the \emph{diastatic function} with respect to $\phi$.
\end{defn}

Observe that when $E$ is a line bundle, $D(x,y)$ agrees with the diastatic function in \eqref{eq: diastatic}. The main point of this work is to deal with the additional difficulties coming from the fact that the diastatic function does not have a simple expression in terms of $\phi$ and $\psi$.

\begin{lem} \label{lem: diastatic}
There exists a coordinate unit disk $U$ centered at $x_0$ such that $D: U \times U \to M_{r \times r}(\mathbb{C})$ is real analytic and $D(x,y)$ has a power series expansion of the form
$$
D(x,y) = \sum_{\alpha,\beta \geq 1} D_{\alpha \overline{\beta}}(x) (y-x)^\alpha (\overline{y} - \overline{x})^\beta
$$
for all $x,y \in U$. Moreover, for any $\alpha,\beta \in \mathbb{Z}_{\geq 0}$,
$$
D_{\alpha\overline{\beta}}(x) = \frac{1}{\alpha! \beta!} {\partial_2}^\alpha {\overline{\partial}_2}^\beta D(x,x)
$$
is real analytic.
\end{lem}

\begin{proof}

Let $U$ be a sufficiently small coordinate unit disk centered at $x_0$ such that $h = e^{-\phi}$ and define $\psi : U \times U \to M_{r \times r}(\mathbb{C})$ by \eqref{eq : psi}. We claim that $\psi$ is the unique matrix-valued function such that $\psi$ is holomorphic on $U \times U$ and $\psi(x,\overline{x}) = \phi(x)$ for any $x \in U$. To see this, suppose that $\psi'$ is another such function. Define $T : \mathbb{C}^2 \to \mathbb{C}^2$ by $T(x,y) = (x+iy,x-iy)$. Then, $(\psi - \psi')\circ T$ is holomorphic in $T^{-1}(U \times U)$ and vanishes when restricted to $\mathbb{R}\times \mathbb{R} \cap T^{-1}(U \times U)$. Because the series expansion of $(\psi - \psi')\circ T$ centered at any point in $\mathbb{R} \times \mathbb{R} \cap T^{-1}(U \times U)$ is zero, $(\psi - \psi') \circ T$ must vanish on an open set containing $\mathbb{R}\times \mathbb{R} \cap T^{-1}(U \times U)$. Then, $\psi - \psi'$ must be identically zero since $U \times U$ is connected. This proves our claim.

By shrinking $U$ if necessary, we may assume that $\phi$ and $D$ have power series expansions
\begin{equation*}
\phi(y) = \sum_{\alpha,\beta \geq 0} \phi_{\alpha \overline{\beta}}(x) (y-x)^\alpha (\overline{y} - \overline{x})^\beta \hspace{10pt} \text{ and } \hspace{10pt} D(x,y) = \sum_{\alpha,\beta \geq 0} D_{\alpha \overline{\beta}}(x) (y-x)^\alpha (\overline{y} - \overline{x})^\beta
\end{equation*}
for all $x,y \in U$ where $\phi_{\alpha\overline{\beta}}(x) = \frac{1}{\alpha! \beta!} \partial^\alpha \overline{\partial}^\beta \phi(x)$ and $D_{\alpha\overline{\beta}}(x) = \frac{1}{\alpha! \beta!} {\partial_2}^\alpha {\overline{\partial}_2}^\beta D(x,x)$. Furthermore, our claim implies that
\begin{equation*}
\psi(y,z) = \sum_{\alpha,\beta \geq 0} \phi_{\alpha \overline{\beta}}(x) (y-x)^\alpha (z - \overline{x})^\beta
\end{equation*}
for any $x,y,z \in U$. Set
$$
H(y)= \sum_{\alpha > 0} \phi_{\alpha0}(x) (y-x)^\alpha \hspace{10pt} \text{ and } \hspace{10pt} 
M(y,\overline{y}) = \sum_{\alpha,\beta > 0} \phi_{\alpha \overline{\beta}}(x) (y-x)^\alpha (\overline{y} - \overline{x})^\beta. 
$$
Then,
\begin{align*}
& \phi(y) = \phi(x) + H(y) + H(y)^* + M(y,\overline{y}), \\
& \psi(y,\overline{x}) = \phi(x) + H(y), \text{ and } \\
& \psi(x,\overline{y}) = \phi(x) + H(y)^*.
\end{align*}
The power series expansion of $D(x,y)$ is obtained by formally composing the power series expansion of $Z^{(5)}$ and the power series expansions of $-\phi(x)/2$, $-\phi(y)$, $\psi(y,\overline{x})$, and $\psi(x,\overline{y})$. It follows that any contribution from $H(y)^*$ or $M(y,\overline{y})$ must have a factor of $(\overline{y}-\overline{x})$. So, the holomorphic part $\sum_{\alpha \geq 0} D_{\alpha,\overline{0}}(x)(y-x)^\alpha$ of $D(x,y)$ is precisely
$$
\log \left ( e^{-\phi(x)/2} e^{\phi(x) + H(y)}e^{-\phi(x) - H(y)}e^{\phi(x)}e^{-\phi(x)/2} \right ) = \log(\Id) = 0.
$$
Similarly, the antiholomorphic part is
$$
\log \left (  e^{-\phi(x)/2} e^{\phi(x)} e^{-\phi(x) - H(y)^*} e^{\phi(x) + H(y)^*} e^{-\phi(x)/2} \right ) = 0.
$$

\end{proof}

In the special case when $(E,h)$ is a holomorphic line bundle with an analytic metric, it was shown in \cite{C} that the diastatic function $D(x,y)$ depends only on the curvature of $h$. Moreover, \eqref{eq: estimate} is a direct consequence of the positivity of the line bundle $(E,h)$. As the lemma below shows, an estimate similar to \eqref{eq: estimate} holds when $E$ has higher rank and $h$ is Griffiths-positive.

\begin{prop} \label{prop: main estimate 1}
For a sufficiently small neighborhood $U$ of $x_0$ and any $x,y \in U$, any eigenvalue $\lambda(x,y)$ of the Hermitian matrix $D(x,y)$ satisfies
$$
\lambda (x,y) \leq -\delta|x-y|^2,
$$ 
for some positive constant $\delta > 0$.
\end{prop}

\begin{proof}
By Lemma \ref{lem: diastatic}, we know that the holomorphic and the antiholomorphic parts of $D(x,y)$ are zero. In particular, the second order term in the power series expansion of $e^{D(x,y)}$ is $D_{1\overline{1}}(x)(y-x)(\overline{y} - \overline{x})$. We can compute the second order terms of $D(x,y)$ directly.
\begin{align*}
D_{1\overline{1}}(x) 
    & = \partial_2 \overline{\partial}_2 \left ( e^{D(x,y)} \right ) \mid_{y = x} \\
    & = \partial_2 \overline{\partial}_2 \left ( e^{-\phi(x)/2} e^{\psi(y,\overline{x})} e^{-\phi(y)} e^{\psi(x,\overline{y})} e^{-\phi(x)/2} \right ) \mid_{y=x} \\
    & = e^{-\phi(x)/2} \partial_1 (e^{\psi(y,\overline{x})}) \overline{\partial}_1 (e^{-\phi(y)}) e^{\psi(x,\overline{y})} e^{-\phi(x)/2}  \mid_{y=x} \\
    & \hspace{10pt} +  e^{-\phi(x)/2} e^{\psi(y,\overline{x})} \partial_1 ( \overline{\partial}_1 (e^{-\phi(y)}) )e^{\psi(x,\overline{y})} e^{-\phi(x)/2}  \mid_{y=x}\\
    & \hspace{10pt} +  e^{-\phi(x)/2} \partial_1(e^{\psi(y,\overline{x})}) e^{-\phi(y)} \overline{\partial}_2 (e^{\psi(x,\overline{y})}) e^{-\phi(x)/2} \mid_{y=x} \\
    & \hspace{10pt} +  e^{-\phi(x)/2} e^{\psi(y,\overline{x})} \partial_1(e^{-\phi(y)}) \overline{\partial}_2 (e^{\psi(x,\overline{y})}) e^{-\phi(x)/2} \mid_{y=x} \\
    & = h^{1/2} (\partial_1(h^{-1})) (\overline{\partial}_1 h) h^{-1}h^{1/2}(x) \\
    & \hspace{10pt} + h^{-1/2} (\partial_1\overline{\partial}_1 h) h^{-1} h^{1/2} (x) \\
    & \hspace{10pt} + h^{1/2} (\partial_1(h^{-1}))h(\overline{\partial}_1(h^{-1}))h^{1/2}(x) \\
    & \hspace{10pt} + h^{-1/2} (\partial_1(h)) (\overline{\partial_1}(h^{-1}))h^{1/2}(x) \\
    & = h^{-1/2} (\overline{\partial}_1 ((\partial_1 h)h^{-1})) h^{1/2}(x).
\end{align*}
Then, 
$$
D_{1\overline{1}}(x) \,d\overline{y} \wedge dy = h^{-1/2} (\overline{\partial} ((\partial h)h^{-1})) h^{1/2}(x) = h^{-1/2}(x) F_{h} (x) h^{1/2}(x)
$$
is the curvature of $h$ with respect to the orthonormal frame given by the rows of $h^{-1/2}(y)$. 

We write $F_h = \tilde{F}  dy \wedge d\overline{y}$. Then,
\begin{align*}
D_{1\overline{1}}(x) (y-x)(\overline{y} - \overline{x}) & = D_{1\overline{1}}(x) \,dy \wedge d\overline{y} \left ( (y-x) \tanv{y}, (\overline{y} - \overline{x}) \tanv{\overline{y}}) \right ) \\
    & = - h^{-1/2}(x)  \tilde{F}  (x)  h^{1/2}(x) (y-x)(\overline{y}-\overline{x}).
\end{align*}
Our assumption that $h$ is Griffiths-positive implies that any eigenvalue of $h^{-1/2}(x)  \tilde{F}(x)  \overline{h}^{1/2}(x)$ is positive. We replace $U$ with a smaller coordinate unit disk such that there exists $M > 0$ with $D \in C_{op,M,5}(x,x)$ for all $x \in U$ and $U$ is compactly contained in another coordinate neighborhood. Then, there exists a constant $\delta > 0$ such that any eigenvalue $\lambda(x,y)$ of $D(x,y)$ must satisfy
$$
\lambda(x,y) \leq  -\delta|x-y|^2 + O(|x-y|^3)
$$
for all $x,y \in U$. By replacing $U$ with a smaller coordinate unit disk and $\delta$ by a smaller positive constant if necessary, we can ensure that
$$
\lambda(x,y) \leq -\delta|x-y|^2
$$
for all $x,y \in U$.
\end{proof}

We introduce some notation. Let $V_1$ and $V_2$ be complex vector spaces equipped with Hermitian inner products $h_1$ and $h_2$, respectively, and let $A : V_1 \to V_2$ be a homomorphism. Recall that the operator norm of $A$ is defined by
\begin{equation*}
\| A \|_{op(V_1,V_2)}^2 := \sup_{v \in V_1, h_1(v,v) = 1} h_2(A(v),A(v)).
\end{equation*}
If $V_1 = \mathbb{C}^{d_1}$, $V_2 = \mathbb{C}^{d_2}$, and $h_1$ and $h_2$ are the standard Hermitian inner products, then we will simply write $\|A\|_{op}$ for the operator norm of $A$.

Now, let $k \in \mathbb{N}$. For any complex vector space $V$, let $\mathfrak{s}^k : \End(V) \to \End(\Sym^k V)$ denote the Lie algebra map corresponding to the homomorphism $\Sym^k : \Aut(V) \to \Aut(\Sym^k V)$. Suppose that $h$ is a Hermitian metric on $V$. Then we denote by $\Sym^k h$ the Hermitian metric on $\Sym^k V$ defined by 
\begin{equation*}
    \Sym^k h( v_1 \dotsb v_k, w_1 \dotsb w_k) = \sum_{\sigma \in S_k} h(v_1,w_{\sigma(1)}) \dotsb h(v_k,w_{\sigma(k)}).
\end{equation*}

Let $\{e_1, \dotsc, e_r\}$ be a local frame of $E$ on $U$. We trivialize $\Sym^k E$ on $U$ using the local frame
\begin{equation*}
\left\{ \frac{e_{1}^{n_1} \dotsb e_{r}^{n_r}}{\sqrt{n_1 !} \dotsb \sqrt{n_r!}} : n_i \geq 0 \text{ and } n_1 + \dotsb + n_r = k\right \}.
\end{equation*}
By replacing $U$ with a sufficiently small coordinate neighborhood if necessary, we can write $h = e^{-\phi}$ for some matrix valued function $\phi$. Then, by Proposition \ref{prop: sym}, we can write $\Sym^k h = e^{-\mathfrak{s}^k(\phi)}$. The proof of the following proposition is included in Appendix A.
\begin{prop} \label{prop: k}
Let $V$ be an inner product space with $\dim(V) = r < \infty$. Then, there exists a constant $C(r)$ that only depend on $r$ such that, for any endomorphism $M \in \End(V)$,
$$
\|\mathfrak{s}^k(M) \|_{op} \leq C(r)k\|M\|_{op}.
$$
Furthermore, if $M$ is Hermitian and the set of eigenvalues of $M$ is $\{\lambda_1, \dotsc, \lambda_r\}$, then $\mathfrak{s}^k M$ is also Hermitian and the set of eigenvalues of $\mathfrak{s}^k(M)$ is $\left \{ \sum_{i=1}^r n_i \lambda_i : \sum_{i=1}^r n_i = k \right \}$. 
\end{prop}

For each $x,y \in U$, we view $e^{-\mathfrak{s}^k(\phi(y))}e^{\mathfrak{s}^k(\psi(x,\overline{y}))}$ as a homomorphism from $\Sym^k E_y$ to $\Sym^k E_x$ acting on row vectors.

\begin{prop} \label{prop: main estimate 2}
Let $x_0 \in X$ and let $U$ be a sufficiently small coordinate unit disk centered at $x_0$. Then, there exists a constant $\delta > 0$ such that
$$
\|e^{-\mathfrak{s}^k( \phi(y))}e^{\mathfrak{s}^k(\psi(x,\overline{y}))}\|_{op(\Sym^kE_y,\Sym^kE_x)}^2 \leq e^{-\delta k|x-y|^2}
$$
for any $x,y \in U$ and $k \in \mathbb{N}$. 
\end{prop}
\begin{proof}

We observe that
\begin{align*}
\|e^{-\mathfrak{s}^k(\phi(y))}e^{\mathfrak{s}^k(\psi(x,\overline{y}))}\|_{op(E_y,E_x)}^2 & = \| e^{-\mathfrak{s}^k(\phi(y))/2} e^{\mathfrak{s}^k(\psi(x,\overline{y}))} e^{-\mathfrak{s}^k(\phi(x))/2}\|_{op}^2 \\
    & = \|e^{-\mathfrak{s}^k(\phi(x))/2} e^{\mathfrak{s}^k(\psi(y,\overline{x}))}e^{ - \mathfrak{s}^k(\phi(y))}  e^{\mathfrak{s}^k(\psi(x,\overline{y}))} e^{-\mathfrak{s}^k(\phi(x))/2}\|_{op} 
\end{align*}
where the second equality is due to the fact that $\|A\|_{op}^2 = \|A^*A\|_{op}$ for any matrix $A$. On the other hand,
\begin{equation*}
\begin{split}
e^{\mathfrak{s}^k \left ( D(x,y) \right )}
& = \Sym^k \left ( e^{D(x,y)} \right ) \\
& = \Sym^k \left ( e^{-\phi(x)/2} e^{\psi(y,\overline{x})}e^{-\phi(y)}  e^{\psi(x,\overline{y})} e^{-\phi(x)/2}\right )  \\
& = \Sym^k \left ( e^{-\phi(x)/2} \right ) \Sym^k \left ( e^{\psi(y,\overline{x})} \right ) \Sym^k \left ( e^{-\phi(y)} \right ) \Sym^k \left (  e^{\psi(x,\overline{y})} \right ) \Sym^k \left ( e^{-\phi(x)/2} \right )\\
& = e^{-\mathfrak{s}^k(\phi(x))/2} e^{\mathfrak{s}^k(\psi(y,\overline{x}))}e^{ - \mathfrak{s}^k(\phi(y))}  e^{\mathfrak{s}^k(\psi(x,\overline{y}))} e^{-\mathfrak{s}^k(\phi(x))/2}
\end{split}
\end{equation*}
for all $x,y \in U$.
As a result,
\begin{equation*}
\|e^{-\mathfrak{s}^k(\phi(y))}e^{\mathfrak{s}^k(\psi(x,\overline{y}))}\|_{op(E_y,E_x)}^2 = \|e^{\mathfrak{s}^k (D(x,y))}\|_{op}
\end{equation*}
for all $x,y \in U$. Let $\lambda_{\max}(x,y)$ denote the maximum eigenvalue of $D(x,y)$. By Proposition \ref{prop: k} and Proposition \ref{prop: main estimate 1}, 
$$
\left \| \mathfrak{s}^k(D(x,y))  \right \|_{op} \leq k \lambda_{\max}(x,y) \leq -\delta k |x-y|^2.
$$
So, for all $x,y \in U$,
$$
\|e^{- \mathfrak{s}^k(\phi(y))} e^{\mathfrak{s}^k(\psi(x,\overline{y}))}\|_{op(\Sym^k E_y, \Sym^k E_x)}^2 \leq e^{-\delta k |x-y|^2}.
$$

\end{proof}

\section{Phase}

In this section, we will work locally in a coordinate unit disk $U$ centered at $x_0$ that satisfies the hypothesis of Proposition \ref{prop: main estimate 2}. Let $(\cdot, \cdot)_{U,k}$ denote the $L^2$ inner product
$$
(\cdot, \cdot)_{U,k} = \int_U \langle \cdot, \cdot \rangle_{\Sym^kh} \omega
$$
and let $\| \cdot \|_{U,k}$ denote the induced norm. We write $H_{U,k}$ for the set of all holomorphic sections $u : U \to \Sym^k E$ such that $(u,u)_{U,k} < \infty$. Choose a smooth function $\chi$ on $U$ that is compactly supported and is identically $1$ on the disk of radius $\frac{1}{2}$ centered at the origin. As in \cite{BBS}, we will construct a local reproducing kernel with an exponential error term.

Before we define what a local reproducing kernel is, we note the following fact. For any complex vector spaces $V$ and $W$, we can extend the inner product $( \cdot, \cdot)_{U,k} : C^\infty(U,\Sym^k E) \times C^\infty(U,\Sym^k E) \to \mathbb{C}$ to a pairing $(\cdot, \cdot)_{U,k} : C^\infty(U,\Sym^k E) \otimes V \times C^\infty(U,\Sym^k E) \otimes W \to \overline{W} \otimes V$ by using the formula
$$
(s \otimes v, t \otimes w)_{U,k} = (s,t)_{U,k} \cdot ( \overline{w} \otimes v),
$$
where $\overline{w}$ is the image of $w$ under the conjugation map $W \to \overline{W}$.
\begin{defn}
We say that smooth sections $\sigma_k : U \times U \to \Sym^k E \boxtimes \Sym^k \overline{E}$ for $k \in \mathbb{N}$ are \emph{reproducing kernels modulo $O(e^{-\delta k})$ for $H_{U,k}$} if 
\begin{equation*}
\left \| u(x) - (\chi u, \sigma_k(\cdot, x))_{U,k} \right \|_{\Sym^k h(x)} = O(e^{-\delta k}) \|u\|_{U,k} 
\end{equation*}
for all $u \in H_{U,k}$ and for all $x$ in a sufficiently small neighborhood of the origin.
\end{defn}

We construct a local reproducing kernel as follows. Define $P : U \times U \times U \to M_{r \times r}(\mathbb{C})$ by
\begin{equation} \label{eq: Phase}
P(x,y,z) = \tilde{Z}(-\psi(y,z), \psi(x,z)).
\end{equation}
Then, $P$ is analytic on $U$ after replacing $U$ with a smaller coordinate unit disk if necessary. Define $\theta : U\times U \times U \to M_{r \times r}(\mathbb{C})$ by
$$
\theta(x,y,z) = \int_0^1 \partial_1 P(tx + (1-t) y, y, z) \,dt. 
$$

\begin{prop} \label{prop: local reprod}
Let $k \in \mathbb{N}$. Then, $\sigma_k : U \times U \to \Sym^k E \boxtimes \Sym^k \overline{E}$ defined by
\begin{equation} \label{eq: reprod ker}
\overline{\sigma_k (y,x)} = \frac{1}{2\pi\sqrt{-1}} e^{\mathfrak{s}^k \psi(x,\overline{y})} \mathfrak{s}^k \left (\sum_{n=1}^\infty \frac{1}{n!} \ad(-\theta(x,y,\overline{y})(x-y))^{n-1} \left (\deriv{\theta}{\overline{y}}(x,y,\overline{y}) \right )\right ) \frac{d\overline{y} \wedge dy}{\omega}(y)
\end{equation}
is a reproducing kernel modulo $O(e^{-\delta k})$ for $H_{U,k}$. In other words,
$$
\left \| u(x) - \left ( \chi u, \sigma_k(\cdot, x) \right )_{U,k}  \right \|_{\Sym^k h(x)} =O (e^{-\delta k})\|u\|_{U,k}
$$
for all $x \in U$ with $|x| < 1/4$ for some constant $\delta >0$.
\end{prop}

\begin{proof}

Let $k \in \mathbb{N}$ and let $u \in H_{U,k}$. Fix $x \in U$ with $|x| < 1/4$. We consider the open set
$$
V = \{(y,\theta) : y \in U \text{ and } \theta \in M_{r \times r}(\mathbb{C})\}.
$$
Set $r_k = \rk(\Sym^k E) = \binom{k+r-1}{r-1}$. Define a $r_k \times r_k$ matrix valued $2$-form $\eta_k$ on $V$ by
$$
\eta_k =  \frac{1}{2\pi \sqrt{-1}}  \chi(y) u(y) e^{\mathfrak{s}^k (\theta (x-y))} \mathfrak{s}^k \left (  \sum_{n=1}^\infty  \frac{1}{n!} \ad(-\theta(x-y))^{n-1} (d\theta)  \right ) \wedge dy
$$
where
\begin{equation*}
    d\theta = 
    \begin{pmatrix}
    d \theta_{11} & \dotsb & d \theta_{1r} \\
    \vdots & \ddots & \vdots \\
    d \theta_{r1} & \dotsb & d\theta_{rr}
    \end{pmatrix}.
\end{equation*}
Note that $\sum_{n=1}^\infty \frac{1}{n!} \ad (-\theta(x-y))^{n-1}$ is holomorphic in $\theta$ because $\sum_{n=1}^\infty \frac{1}{n!} z^{n-1} = \frac{e^z -1}{z}$ is an entire function. For each $s \geq 0$, define $\Lambda_s \subseteq V$ by
$$
\Lambda_s = \{(y,\theta(x,y,\overline{y})-s(\overline{x} - \overline{y})) : y \in U\}
$$
and define $F_s : U \times [0,s] \to V$ by
$$
F_s(y, t) = (y,\theta(x,y,\overline{y}) - t(\overline{x} - \overline{y})).
$$
By Stoke's Theorem,
$$
\int_{\Lambda_s} \eta_k - \int_{\Lambda_0} \eta_k = \int_{U \times (0,s)} d ({F_s}^*\eta_k).
$$
We set $T = \frac{1}{r} \tr \theta$, $\theta_0 = \theta - T \Id$, and $d\Vol = \frac{\sqrt{-1}}{2} dy \wedge d\overline{y}$. Note that $\ad(\theta_0) = \ad(\theta)$ since $\theta - \theta_0$ is a multiple of the identity matrix. Then,
\begin{align*}
& \hspace{10pt} \int_{\Lambda_s} \eta_k  \\
    & = \int_{\Lambda_s} \frac{1}{2\pi \sqrt{-1}} \chi(y)u(y)  e^{\mathfrak{s}^k(\theta(x-y))} \mathfrak{s}^k\left ( \sum_{n=1}^\infty \frac{1}{n!} \ad(-\theta_0(x-y))^{n-1} (d\theta_0) \right ) \wedge dy \\
    & \hspace{10pt} + \int_{\Lambda_s} \frac{1}{2\pi \sqrt{-1}} \chi(y)u(y) e^{\mathfrak{s}^k(\theta(x-y))} k\,dT \wedge \,dy \\
    & = \int_{U} \frac{1}{\pi} e^{-ks|x-y|^2} \chi(y)u(y)  e^{\mathfrak{s}^k(\theta(x,y,\overline{y})(x-y))}  \mathfrak{s}^k\left ( \sum_{n=1}^\infty \frac{1}{n!} \ad(-\theta_0(x,y,\overline{y}) (x-y))^{n-1} \left ( \deriv{\theta_0}{\overline{y}}(x,y,\overline{y}) \right )  \right ) \,d\Vol \\
    & \hspace{10pt} + \int_{U} \frac{k}{\pi}e^{-ks|x-y|^2} \chi(y)u(y) e^{\mathfrak{s}^k(\theta(x,y,\overline{y})(x-y))} \deriv{T}{\overline{y}}(x,y,\overline{y}) \,d\Vol \\
    & \hspace{10pt} + \int_{U} \frac{ks}{\pi}e^{-ks|x-y|^2} \chi(y)u(y) e^{\mathfrak{s}^k(\theta(x,y,\overline{y})(x-y))} \,d\Vol.
\end{align*}
Note that the functions $\theta_0(x,y,\overline{y})$ and $T(x,y,\overline{y})$ do not depend on $s$. Set
\begin{equation*}
I_1 =  \int_{U} \frac{1}{\pi} e^{-ks|x-y|^2} \chi(y)u(y)  e^{\mathfrak{s}^k(\theta(x,y,\overline{y})(x-y))}  \mathfrak{s}^k\left ( \sum_{n=1}^\infty \frac{1}{n!} \ad(-\theta_0(x,y,\overline{y}) (x-y))^{n-1} \left ( \deriv{\theta_0}{\overline{y}}(x,y,\overline{y}) \right )  \right ) \,d\Vol
\end{equation*}
\begin{equation*}
I_2 = \int_{U} \frac{k}{\pi}e^{-ks|x-y|^2} \chi(y)u(y) e^{\mathfrak{s}^k(\theta(x,y,\overline{y})(x-y))} \deriv{T}{\overline{y}}(x,y,\overline{y}) \,d\Vol,
\end{equation*}
and
\begin{equation*}
I_3 =  \int_{U} \frac{ks}{\pi}e^{-ks|x-y|^2} \chi(y)u(y) e^{\mathfrak{s}^k(\theta(x,y,\overline{y})(x-y))} \,d\Vol.
\end{equation*}
The integral kernel $y \mapsto \frac{ks}{\pi} e^{-ks|x-y|^2}$ defined on $\mathbb{R}^2$ has the property that
\begin{equation*}
\lim_{s \to \infty} \int_{U} \frac{ks}{\pi} e^{-ks|x-y|^2} \varphi(y) \,d\Vol = \varphi(x)
\end{equation*}
for any smooth function $\varphi$ compactly supported in $U$. Then, 
\begin{gather*}
\lim_{s \to \infty} s I_1 = \frac{1}{k}u(x) \mathfrak{s}^k \left ( \deriv{\theta_0}{\overline{y}}(x,x,\overline{x}) \right ), \\
\lim_{s \to \infty} sI_2 = u(x)\deriv{T}{\overline{y}}(x,x,\overline{x}), \hspace{15pt} \text{and} \\
\lim_{s \to \infty} I_3 = u(x).
\end{gather*}
It follows that
$$
\lim_{s \to \infty} \int_{\Lambda_s} \eta_k = \lim_{s \to \infty} I_1 + I_2 + I_3 = u(x).
$$
Moreover, by the definition of $\sigma_k$,
$$
\int_{\Lambda_0} \eta_k   = \left ( \chi u, \sigma_k(\cdot, x) \right )_{U,k}.
$$
So, it is enough to show that $v_{k,s}(x) = \int_{U \times (0,s)} d {F_s}^*\eta_k$ satisfies the bound 
$$
\|v_{k,s}(x)\|_{\Sym^k h(x)} \leq C e^{-\delta k} \|u\|_{U,k}
$$
where $C$ and $\delta$ are constants that do not depend on $k$ and $s$.

If $C(t)$ is a smooth curve in $M_{r\times r}(\mathbb{C})$, then it is well known that 
\begin{equation} \label{eqn: expderiv}
e^{-C(t)} \frac{d}{dt} \left ( e^{C(t)} \right ) = \sum_{n=1}^\infty \frac{1}{n!} \ad(-C(t))^{n-1} (C'(t)).
\end{equation}
For a proof, see \cite{St}. Equation \eqref{eqn: expderiv} implies that 
$$
d \left (  e^{\mathfrak{s}^k (\theta (x-y))} \mathfrak{s}^k \left (  \sum_{n=1}^\infty  \frac{1}{n!} \ad(-\theta(x-y))^{n-1} (d\theta)  \right ) \wedge dy \right ) = d \left ( \frac{d(e^{\mathfrak{s}^k(\theta (x-y))})}{x-y} \wedge dy \right ) = 0.
$$
Additionally, $d u \wedge dy = 0$ because $u$ is holomorphic. Then,
\begin{align*}
d\eta_k & =   \frac{1}{2\pi \sqrt{-1} }  d \chi(y) \wedge u(y) e^{\mathfrak{s}^k (\theta (x-y))} \mathfrak{s}^k \left (  \sum_{n=1}^\infty  \frac{1}{n!} \ad(-\theta(x-y))^{n-1} (d\theta)  \right ) \wedge dy.
\end{align*}
It follows that
\begin{align*}
& \int_{U \times (0,s)} d {F_s}^*\eta_k \\
    & = \frac{1}{2\pi \sqrt{-1}} \int_{U \times (0,s)} d \chi(y) \wedge u(y)  e^{\mathfrak{s}^k (\theta (x-y) - t|x-y|^2\Id)}  \mathfrak{s}^k \left (  \sum_{n=1}^\infty  \frac{1}{n!} \ad(-\theta (x-y))^{n-1} (d( \theta - t(\overline{x} - \overline{y})\Id))  \right ) \wedge  dy \\
    & = \frac{k}{\pi} \int_{U \times (0,s)} \deriv {\chi}{\overline{y}} (y)  u(y)  (\overline{x} - \overline{y}) e^{\mathfrak{s}^k (\theta (x-y))} e^{-kt|x-y|^2} \,d \Vol \wedge dt .
\end{align*}

Note that
\begin{equation*}
    \theta(x,y,\overline{y}) (x-y) =  \int_0^1 \partial_1 P(tx + (1-t) y, y, z) (x-y) \,dt = P(x,y,z).
\end{equation*}
We write
\begin{gather*}
(M^i_j(x,y))= e^{-\mathfrak{s}^k(\phi(y))/2} e^{\mathfrak{s}^k(\psi(x,\overline{y}))} e^{-\mathfrak{s}^k(\phi(x))/2}, \text{ and }\\
(\tilde{u}^j(y))  = u(y)e^{-\mathfrak{s}^k(\phi(y))/2}.
\end{gather*}
Then,
\begin{align*}
\|v_{k,s}(x)\|_{\Sym^k h(x)}^2   
    &\leq \frac{k^2}{\pi^2}\sum_{i=1}^{r_k}  \left ( \int_{U \times (0,s)} \left | \sum_{j=1}^{r_k}\tilde{u}^j(y) M^i_j(x,y) \right | \left |\deriv {\chi}{\overline{y}} (y) e^{-kt|x-y|^2} (\overline{x} - \overline{y}) \right | \,d\Vol\right )  ^2 \\
    & \leq  \frac{k^2}{\pi^2} \sum_{i=1}^{r_k} \left(   \int_{U \times (0,s)} \left | \sum_{j=1}^{r_k}\tilde{u}^j(y) M^i_j(x,y) \right |^2 \left |\deriv {\chi}{\overline{y}} (y) \right |  e^{-kt|x-y|^2} |x-y|  \,d\Vol dt  \right )\\
    & \left ( \int_{U \times (0,s)} \left |\deriv {\chi}{\overline{y}} (y) \right | e^{-kt|x-y|^2} |x-y|  \,d\Vol dt \right ) \\
    & \text{ by H\"older's inequality} \\
    & \leq  \frac{k^2}{\pi^2} \left(   \int_{U \times (0,s)} \left \|(M^i_j(x,y)) \right \|^2_{op} \|u(y)\|_{\Sym^k h(y)}^2  \left |\deriv {\chi}{\overline{y}} (y) \right |  e^{-kt|x-y|^2} |x-y|  \,d\Vol dt  \right )\\
    & \left ( \int_{U \times (0,s)} \left |\deriv {\chi}{\overline{y}} (y) \right | e^{-kt|x-y|^2} |x-y|  \,d\Vol dt \right ) \\
\end{align*}

%We let $M^i(x,y)$ denote the dual element %corresponding to the $i$-th column of %$(M^i_j(x,y))$. So,
%\begin{align*}
%\|v_{k,s}(x)\|_{\mathfrak{s}^k\phi(x)}^2
%    & \leq \frac{k^2}{\pi^2} \sum_i \left %( \int_{U \times (0,s)} \|M^i(x,y)\|_{op} %\|u(y)\|_{\mathfrak{s}^k\phi(y)} \left %|\deriv {\chi}{\overline{y}} (y) \right | %e^{-kt|x-y|^2} |x-y|  \,d\Vol dt \right %)^2 \\
%    & \leq  \frac{r_k}{\pi^2}\left (  %\int_{U \times [0,s]} \|(M^i_j)(x,y)\|_{op}^2 %\|u(y)\|_{\mathfrak{s}^k\phi(y)}^2 \left |\deriv {\chi}{\overline{y}} (y) \right |  e^{-kt|x-y|^2} |x-y|  \,d\Vol dt  \right )\\
%    & \left ( \int_{U \times [0,s]} \left |\deriv {\chi}{\overline{y}} (y) \right | e^{-kt|x-y|^2} |x-y|  \,d\Vol dt \right ) \\
%    & \text{ by H\"older's inequality}.
%\end{align*}

By Proposition \ref{prop: main estimate 2}, $\|(M^i_j(x,y))\|_{op}^2 \leq e^{-\delta k|x-y|^2}$. Since $\deriv{\chi}{\overline{y}} (y) = 0$ for all $|y| < 1/2$ and $|x| < 1/4$, $|x-y|$ must be at least $1/4$ whenever $\deriv{\chi}{\overline{y}} \neq 0$. Additionally, there exists a constant $C > 0$ such that
$$
\int_{\text{supp}(\chi)}  \langle u(y), u(y)\rangle_{\Sym^kh} \,d\Vol  \leq C \|u\|_{U,k}^2.
$$
because $\chi$ is compactly supported in $U$. Then, 
\begin{align*}
& \frac{k^2}{\pi^2} \int_{U \times (0,s)} \|(M^i_j(x,y))\|_{op}^2 \|u(y)\|_{\Sym^k h(y)}^2 \left |\deriv {\chi}{\overline{y}} (y) \right | e^{-kt|x-y|^2} |x-y| \,d\Vol dt  \\
    \leq \hspace{2pt} & C' k^{2} e^{-\delta k/16} \int_{\text{supp}(\chi) \times (0,s)}  \|u(y)\|_{\Sym^k h(y)}^2 e^{-kt/16} \,d\Vol \,dt \\
     \leq \hspace{2pt} & C' k^{2} e^{-\delta k/16} \int_{\text{supp}(\chi)}  \|u(y)\|_{\Sym^k h(y)}^2  \,d\Vol\int_0^\infty e^{-kt/16} \,dt \\
    \leq \hspace{2pt} &  C'' e^{-\delta'k} \|u\|_{U,k}^2
\end{align*}
where $C'' > 0$ is independent of $k$ and $s$ and $\delta' > 0$ is a constant less than $\delta$. Similarly,
$$
\int_{U \times (0,s)} \left | \deriv{\chi}{\overline{y}} \right | e^{-kt|x-y|^2} |x-y| \,d\Vol \,dt    
    \leq  \int_0^\infty \int_{U} 2C' e^{-kt/16} \,d\Vol \,dt 
    \leq  2C'''.
$$
We conclude that 
$$
\|v_{k,s}(x)\|_{\Sym^k h (x)}^2 \leq C'''' e^{-\delta'k} \|u\|_{U,k}^2
$$
as required.

\end{proof}

\section{Negligible Amplitudes}

We have constructed reproducing kernels modulo $O(e^{-\delta k})$ for $H_{U,k}$. However, unlike the Bergman kernel $K_k(y,x)$, the integral kernel $\sigma_k(y,x)$ is not holomorphic in the variable $y$. To correct this, we again follow the ideas in \cite{BBS} and multiply $\sigma_k(y,x)$ by factors that give sufficiently small contribution to the integral.

As above, we set $r_k = \rk(\Sym^k E) = \binom{k+r-1}{r-1}$.  

\begin{defn}
We say that an $r_k \times r_k$ matrix valued function $a_k(x,y,z)$ for $k \in \mathbb{N}$ is a \emph{negligible amplitude} if we can write
\begin{equation*}
\begin{split}
& e^{\mathfrak{s}^k(\theta (x,y,z)(x-y))} \mathfrak{s}^k \left ( \sum_{n=1}^\infty \frac{1}{n!} \ad \left (- \theta(x,y,z) (x-y) \right )^{n-1} \left ( \deriv{\theta}{z} (x,y,z)\right )\right )a_k(x,y,z) \,dz \wedge dy \\ 
= \hspace{2pt} & d (e^{\mathfrak{s}^k (\theta(x,y,z) (x-y))} A_k(x,y,z) \,dy)
\end{split}
\end{equation*}
where $A_k(x,y,z) \,dy$ is an $r_k \times r_k$ matrix valued one form.
\end{defn}

Equivalently, $a_k$ is a negligible amplitude if and only if 
\begin{equation*}
    \begin{split}
    &   \mathfrak{s}^k \left ( \sum_{n=1}^\infty \frac{1}{n!} \ad \left (-\theta(x,y,z) (x-y) \right )^{n-1} \left ( \deriv{\theta}{z}(x,y,z) \right )\right ) a_k(x,y,z) \\
        = \hspace{2pt} & \deriv{A_k}{z}(x,y,z) + \mathfrak{s}^k \left ( \sum_{n=1}^\infty \frac{1}{n!} \ad \left (-\theta (x,y,z)(x-y) \right )^{n-1} \left ( \deriv{\theta}{z}(x,y,z)(x-y) \right )\right ) A_k(x,y,z)
    \end{split}
\end{equation*}
for some $r_k \times r_k$ matrix valued function $A_k(x,y,z)$. 

\begin{defn}
We say that an $r_k \times r_k$ matrix-valued function $a^{(N)}_k(x,y,z)$ is a \emph{negligible amplitude modulo $O(k^{-N})$} if we can write
\begin{equation*}
    \begin{split}
    & e^{\mathfrak{s}^k (\theta (x-y))} \mathfrak{s}^k \left ( \sum_{n=1}^\infty \frac{1}{n!} \ad \left (- \theta (x,y,z) (x-y) \right )^{n-1} \left ( \deriv{\theta}{z} (x,y,z) \right )\right )a^{(N)}_k(x,y,z) \,dz \wedge dy \\
        = \hspace{2pt} & d (e^{\mathfrak{s}^k (\theta (x,y,z) (x-y))} A^{(N)}_k(x,y,z) \,dy) + e^{\mathfrak{s}^k(\theta (x,y,z) (x-y))} G_{k,N}(x,y,z) \,dz \wedge dy
    \end{split}
\end{equation*}
where $\|A^{(N)}_k(x,y,z)\|_{op(\Sym^kE_x,\Sym^kE_x)} = O(1)$ and $\|G_{k,N}(x,y,z)\|_{op(\Sym^kE_x,\Sym^kE_x)} = O(k^{-N})$.
\end{defn}

We will show that negligible amplitudes give sufficiently small contributions to the integral.

\begin{lem} \label{lem: Operator}
Let $\Gamma$ be a Riemannian manifold with finite volume and let $A: \Gamma \to M_{r\times r}(\mathbb{C})$ be a continuous matrix valued function on $\Gamma$. Then,
\begin{equation*}
\left \| \int_\Gamma A(x) \,dx \right \|_{op} \leq \Vol(\Gamma)^{1/2} \left ( \int_\Gamma \| A(x)\|_{op}^2 \,dx \right)^{1/2}.
\end{equation*}
\end{lem}

\begin{proof}
Let $v = (v^1, \dotsc, v^r) \in \mathbb{C}^r$. We observe that
\begin{align*}
\left \| \left ( \int_\Gamma A(x) \,dx\right ) (v) \right \|^2 & = \sum_{i=1}^r \left | \int_\Gamma \sum_{j=1}^r A_j^i(x)v^j \,dx \right |^2 \\
    & \leq \sum_{i=1}^r \int_\Gamma \left | \sum_{j=1}^r A_j^i(x)v^j \right |^2 \,dx \int_\Gamma 1 \,dx \\
    & \text{ by H\"older's inequality} \\
    & = \Vol(\Gamma) \int_\Gamma \sum_{i=1}^r \left | \sum_{j=1}^r A_j^i(x)v^j \right |^2 \,dx \\
    & = \Vol(U) \int_U \|(A(x))(v)\|^2 \,dx \\
    & \leq \Vol(U) \left ( \int_U \|A(x)\|_{op}^2 \,dx \right ) \|v\|^2.
\end{align*}
It follows from the definition of $\|\cdot\|_{op}$ that
\begin{equation*}
\left \| \int_\Gamma A(x) \,dx \right \|_{op}^2 \leq \Vol(\Gamma) \int_\Gamma \| A(x)\|_{op}^2 \,dx.
\end{equation*}
\end{proof}

\begin{thm} \label{thm: local reprod mod}
Let $x_0 \in X$ and $N \in \mathbb{N}$. Let $U$ be a coordinate unit ball centered at $x_0$ small enough so that $\sigma_k : U \times U \to \Sym^kE \boxtimes \Sym^k \overline{E}$ defined by \eqref{eq: reprod ker} is a reproducing kernel modulo $O(e^{-\delta k})$ for $H_{U,k}$. Suppose that $a_k^{(N)}: U \times U \times U \to M_{r_k \times r_k}(\mathbb{C})$ is a negligible amplitude modulo $O(k^{-N})$. Then,
$$
\left \|u(x) - \left (\chi u, \sigma_{k}(\cdot, x) \left ( \overline{\Id + a_k^{(N)}(x,\cdot,\overline{\cdot})} \right ) \right )_{U,k} \right  \|_{\Sym^k h(x)} = O(k^{-N})   \|u\|_{U,k}
$$
for every $u \in H_{U,k}$, and $x \in U$ with $|x| < 1/4$.
\end{thm}

\begin{proof}
Proposition \ref{prop: local reprod} implies that $\left \| u(x) - (\chi u,\sigma_k(\cdot,x))_{U,k} \right \|_{\Sym^k h(x)} = O(e^{-\delta k})\|u\|_{U,k}$. So, it suffices to show that $\left \|\left  (\chi u, \sigma_k(\cdot,x) \overline{ a^{(N)}_k(x,\cdot ,\overline{\cdot})} \right )_{U,k} \right \|_{\Sym^k h(x)} = O(k^{-N}) \|u\|_{U,k}$. We see that
\begin{align*}
{\left (\chi u, \sigma_k(\cdot,x) \overline{ a^{(N)}_k(x,\cdot ,\overline{\cdot})} \right )_{U,k}}
    = \hspace{2pt} & \int_{U} \frac{1}{2\pi \sqrt{-1}} \chi(y) u(y) \, d \left ( e^{\mathfrak{s}^k (\theta(x,y,\overline{y})(x-y))} A^{(N)}_k(x,y,\overline{y}) \right ) \wedge dy \\
    &  - \int_{U} \frac{1}{2\pi \sqrt{-1}}  \chi(y) u(y) \,  e^{\mathfrak{s}^k (\theta(x,y,\overline{y})(x-y))} G_{k,N}(x,y,\overline{y}) \,d\overline{y} \wedge dy \\
    = \hspace{2pt} &  - \int_U \frac{1}{2\pi \sqrt{-1}} d \chi(y) \wedge u(y) e^{\mathfrak{s}^k (\theta(x,y,\overline{y})(x-y))} A^{(N)}_k (x,y,\overline{y}) \, dy \\
    &  - \int_{U} \frac{1}{2\pi \sqrt{-1}} \chi(y) u(y) \,  e^{\mathfrak{s}^k (\theta(x,y,\overline{y})(x-y))} G_{k,N}(x,y,\overline{y}) \, d\overline{y} \wedge dy \\
    = \hspace{2pt} &  -\frac{1}{\pi} \int_U \deriv{\chi}{\overline{y}}(y) u(y) e^{\mathfrak{s}^k(\theta(x,y,\overline{y})(x-y))}A^{(N)}_k(x,y,\overline{y}) \,d\Vol \\
    & - \frac{1}{\pi} \int_{U} \chi(y) u(y) \,  e^{\mathfrak{s}^k (\theta(x,y,\overline{y})(x-y))} G_{k,N}(x,y,\overline{y}) \,d \Vol. 
\end{align*}
We set $E^k := \Sym^k E$ and observe that,
\begin{align*}
& \left \|  \int_U \deriv{\chi}{\overline{y}}(y) u(y) e^{\mathfrak{s}^k(\theta(x,y,\overline{y})(x-y))}A^{(N)}_k(x,y,\overline{y}) \,d\Vol \right \|_{\Sym^k h(x)}^2  \\
    \leq \hspace{2pt} & \Vol(U) \int_U \left |\deriv{\chi}{\overline{y}}(y) \right |^2 \| u(y) \|_{\Sym^k h(y)}^2  \left \| e^{\mathfrak{s}^k(\theta(x,y,\overline{y})(x-y))}\right\|_{op( E_y^k, E_x^k)}^2  \left \| A^{(N)}_k(x,y,\overline{y})\right \|_{op( E_x^k, E_x^k)}^2 \,d \Vol \\
    & \text{ by Lemma \ref{lem: Operator}} \\
    \leq \hspace{2pt} & C e^{-\delta k} \int_{\text{supp}(\chi)} \| u(y) \|_{\Sym^k h(y)}^2 \\
    & \text{ by Proposition \ref{prop: main estimate 2} and the fact that $\left \| A^{(N)}_k(x,y,\overline{y})\right \|_{op( E_x^k, E_x^k)} = O(1)$ } \\
    \leq  \hspace{2pt} & C'  e^{-\delta k}\|u\|_{U,k}^2.
\end{align*}
Similarly,
$$
 \left \|  \int_{U} \chi(y) u(y) \,  e^{\mathfrak{s}^k (\theta(x-y))} G_{k,N}(x,y,\overline{y}) \right \|_{\Sym^k h(x)}^2 = O(k^{-N}) \|u\|_{U,k}^2. 
$$

\end{proof}

\section{Construction of Negligible Amplitudes}

In this section, we will construct negligible amplitudes modulo $O(k^{-N})$ on a sufficiently small coordinate unit disk $U$ centered at $x_0 \in X$. The negligible amplitudes $a_k^{(N)}$ will have the property that $\sigma_k(y,x)\left ( \overline{\Id + a_k^{(N)}(x,y,\overline{y})} \right )$ is analytic in $\overline{x}$ and $y$. 

Set $E^k := \Sym^k E$. We look for an analytic matrix valued function $b_k$ so that
$$
\mathfrak{s}^k \left ( \sum_{n=1}^\infty \frac{1}{n!} \ad \left (-\theta (x,y,z)(x-y) \right )^{n-1} \left ( \deriv{\theta}{z} (x,y,z) \right )\right ) (\Id + a_k(x,y,z)) =   b_k(x,z) \tilde{g}(y,z)
$$
where
$$
\tilde{g}(y,z) =  \partial_1\partial_2 \varphi(y,z)
$$
and $\varphi$ is a real analytic function on $U$ such that $\sqrt{-1} \partial \overline{\partial}\varphi = \omega \mid_U$. We write
\begin{equation} \label{eq: tau}
\tau(x,y,z) := \sum_{n=1}^\infty \frac{1}{n!} \ad (-\theta(x,y,z)(x-y))^{n-1} \left ( \deriv{\theta(x,y,z)}{z}\right ).
\end{equation}
Then, the equations that we wish to solve become
\begin{equation} \label{eq: power series}
\begin{gathered}
\mathfrak{s}^k(\tau(x,y,z)) a_k(x,y,z)  = \deriv{A_k}{z}(x,y,z) + \mathfrak{s}^k (\tau(x,y,z))(x-y) A_k(x,y,z)  \hspace{10pt} \text{ and } \\
\mathfrak{s}^k \left ( \tau(x,y,z) \right ) (\Id + a_k(x,y,z)) =   b_k(x,z) \tilde{g}(y,z) .
\end{gathered}
\end{equation}
We view each matrix involved in the equation as an endomorphism of $E^k_x$ acting on row vectors.

We look for formal power series solutions
\begin{gather*}
a_k(x,y,z) = \sum_{m=0}^\infty \frac{a_{k,m}(x,y,z)}{k^m}, \\
A_k(x,y,z) = \sum_{m=0}^\infty \frac{A_{k,m}(x,y,z)}{k^m} , \text{ and } \\
b_k(x,z) = k\sum_{m=0}^\infty \frac{b_{k,m}(x,z)}{k^m},
\end{gather*}
where $a_{k,m}(x,y,z)$, $A_{k,m}(x,y,z)$, and $b_k(x,z)$ are analytic matrix-valued functions. The series above need not converge, but we require that, for $x,y,z$ in a small coordinate disk centered at $x_0$, the quantities $\|a_{k,m}(x,y,z)\|_{op(E^k_x,E^k_x)}$, $\|A_{k,m}(x,y,z)\|_{op(E^k_x,E^k_x)}$,  \text{ and} $\|b_{k,m}(x,z)\|_{op(E^k_x,E^k_x)}$ are bounded from above by constants independent of $k$ for all $k \in \mathbb{N}$ and $m \in \mathbb{Z}_{\geq 0}$.

By \eqref{eq: tau}, \eqref{eqn: expderiv} and \eqref{eq: Phase},
\begin{equation} \label{eq: tau2}
\tau(x,y,z) =  \frac{ e^{-P(x,y,z)} \partial_3 e^{P(x,y,z)}} {x-y} 
    = \frac{e^{-\psi(x,z)}e^{\psi(y,z)} \partial_2 \left ( e^{-\psi(y,z)} e^{\psi(x,z)} \right )}{ x-y}.
\end{equation}
Using \eqref{eq: tau2}, one can show that
$$
\tau(0,0,0) = \tilde{F}(0,0)
$$
where $\tilde{F}$ is the matrix-valued function on $U \times U$ such that $\tilde{F} (y,\overline{y})\,dy \wedge d \overline{y}$ is the curvature of $h$ at $y$. Since $\mathfrak{s}^k \left ( \tilde{F}(0,0) \right )$ is a positive definite Hermitian matrix, $\mathfrak{s}^k(\tau(x,y,z))$ is invertible near the origin. Similarly, $\tilde{g}(y,z)$ also has a local inverse because 
$$
\tilde{g}(0,0) = \sqrt{-1} \frac{\omega(x_0)}{d\overline{y} \wedge dy} \neq 0.
$$
Additionally, we observe that
\begin{equation*}
\begin{aligned}
\|\mathfrak{s}^k(\tau(x,y,z))\|_{op(E^k_x,E^k_x)} = & \|e^{\mathfrak{s}^k (\phi(x))/2}\mathfrak{s}^k(\tau(x,y,z)) e^{-\mathfrak{s}^k (\phi(x))/2}\|_{op} \\
    = \hspace{2pt} & \|\Sym^k (e^{\phi(x)/2})\mathfrak{s}^k(\tau(x,y,z)) \Sym^k(e^{-\phi(x)/2}) \|_{op} \\
    = \hspace{2pt} & \|\mathfrak{s}^k(e^{\phi(x)/2} \tau(x,y,z)e^{-\phi(x)/2})\|_{op} \\
    = \hspace{2pt} & O(k).
\end{aligned}
\end{equation*}

The terms of order $k$ in \eqref{eq: power series} are
\begin{gather*}
\mathfrak{s}^k(\tau(x,y,z)) a_{k,0}(x,y,z)  = \mathfrak{s}^k (\tau(x,y,z))(x-y) A_{k,0}(x,y,z)  \hspace{10pt} \text{ and } \\
\mathfrak{s}^k \left ( \tau(x,y,z) \right )(\Id + a_{k,0}(x,y,z)) =   k b_{k,0}(x,z) \tilde{g}(y,z).
\end{gather*}
Since $\mathfrak{s}^k(\tau(x,y,z))$ has a local inverse, the first equation is equivalent to 
$$
a_{k,0}(x,y,z) = (x-y) A_{k,0}(x,y,z).
$$
Then, restricting the second equation to $\{(x,y,z) \in U\times U \times U: y=x\}$ yields
$$
\mathfrak{s}^k \left ( \deriv{\theta}{z}(x,x,z) \right ) = k  b_{k,0}(x,z) \tilde{g}(x,z).
$$
Since $\tilde{g}(y,z)$ has a local inverse, 
$$
b_{k,0}(x,z) = \frac{1}{k} \mathfrak{s}^k\left ( \deriv{\theta}{z}(x,x,z) \right ) \tilde{g}(x,z)^{-1}.
$$
Then,
$$
a_{k,0}(x,y,z) =  \mathfrak{s}^k \left ( \tau(x,y,z) \right )^{-1} \mathfrak{s}^k\left ( \deriv{\theta}{z}(x,x,z) \right ) \tilde{g}(x,z)^{-1} \tilde{g}(y,z)  - \Id.
$$
In particular, $a_{k,0}(x,x,z) = 0$. Then, we can find a smaller coordinate unit disk $U_1$ such that $a_{k,0}(x,y,z)(x-y)^{-1}$ is uniformly bounded for all $x,y,z \in U_1$ with $x \neq y$. By the Riemann Extension Theorem (\cite{H}, Proposition 1.1.7), $A_{k,0}(x,y,z)$ is analytic on $U_1 \times U_1 \times U_1$.

Proceeding inductively, suppose that $m > 0$ and that $A_{k,m-1}(x,y,z)$ is analytic on $U_{m-1} \times U_{m-1} \times U_{m-1}$. The terms of order $k^{-m+1}$ in \eqref{eq: power series} are
\begin{gather*}
\mathfrak{s}^k(\tau(x,y,z)) \frac{a_{k,m}(x,y,z)}{k^m}  = \frac{1}{k^{m-1}}\deriv{A_{k,m-1}}{z}(x,y,z) + \mathfrak{s}^k (\tau(x,y,z))(x-y) \frac{A_{k,m}(x,y,z)}{k^m} \text{ and } \\
\mathfrak{s}^k \left ( \tau(x,y,z) \right ) \frac{a_{k,m}(x,y,z)}{k^m} =  \tilde{g}(y,z) \frac{b_{k,m}(x,z)}{k^{m-1}} .
\end{gather*}
Restricting both equations to $\{(x,y,z)\in U_{m-1} \times U_{m-1} \times U_{m-1} : y=x\}$, we see that
$$
\deriv{A_{k,m-1}}{z}(x,x,z) = \frac{1}{k}\mathfrak{s}^k \left ( \deriv{\theta}{z}(x,x,z) \right ) a_{k,m}(x,x,z) =\tilde{g}(x,z) b_{k,m}(x,z)
$$
So, 
$$
b_{k,m}(x,z) = \tilde{g}(x,z)^{-1} \deriv{A_{k,m-1}}{z}(x,x,z).
$$
We can solve for $a_{k,m}$ using the equation
$$
a_{k,m}(x,y,z) =  k \tilde{g}(y,z) \mathfrak{s}^k \left ( \tau(x,y,z) \right )^{-1} b_{k,m}(x,z).
$$
Then, the first equation implies that
$$
A_{k,m}(x,y,z)(x-y) =  a_{k,m}(x,y,z)  - k\mathfrak{s}^k \left ( \tau(x,y,z) \right )^{-1} \deriv{A_{k,m-1}}{z}(x,y,z).
$$
Note that $a_{k,m}(x,x,z) - k\mathfrak{s}^k(\tau(x,x,z))^{-1} \deriv{A_{k,m-1}}{z}(x,x,z) = 0$ for all $x,z \in U_{m-1}$. Applying the Riemann Extension Theorem again, we can find a smaller coordinate unit disk $U_m$ centered at $x_0$ such that $A_{k,m}(x,y,z)$ is analytic on $U_m \times U_m \times U_m$. 

Fix $N \in \mathbb{Z}_{\geq 0}$. For each $m \leq N$, we have constructed functions $a_{k,m}(x,y,z)$, $A_{k,m}(x,y,z)$ and $b_{k,m}(x,z)$ that are analytic on $U\times U \times U$ or $U\times U$ where $U$ is a sufficiently small coordinate unit disk centered at $x_0$. Define $a^{(N)}_k:U \times U \times U \to M_{r_k \times r_k}(\mathbb{C})$, $A^{(N)}_k : U \times U \times U \to M_{r_k \times r_k}(\mathbb{C})$, and $b_k^{(N)} : U \times U \to M_{r_k \times r_k}(\mathbb{C})$ by 
\begin{gather*}
a^{(N)}_k(x,y,z) = \sum_{m=0}^N  \frac{a_{k,m}(x,y,z)}{k^m},  \\
A^{(N)}_k(x,y,z) = \sum_{m=0}^N  \frac{A_{k,m}(x,y,z)}{k^m},\hspace{10pt} \text{ and } \\
b_k^{(N)}(x,z) = k \sum_{m=0}^N \frac{b_{k,m}(x,z)}{k^{m}}.
\end{gather*}
By construction,
\begin{equation} \label{eqn: b_k id}
\sigma_k(y,x)\left ( \overline{\Id + a_k^{(N)}(x,y,\overline{y})} \right ) = \frac{1}{2\pi} \overline{ e^{\mathfrak{s}^k(\psi(x,\overline{y}))}b_k^{(N)}(x,\overline{y})} .
\end{equation}
We will show that $a^{(N)}_k$ is a negligible amplitude modulo $O(k^{-N})$.

\begin{lem} \label{lem: inverse}
For any multi-index $\alpha \in \left (\mathbb{Z}_{\geq 0} \right )^n$, let $A_\alpha$ be an $r$ by $r$ matrix with $\|A_\alpha\|_{op} \leq M \frac{|\alpha|!}{\alpha!} \rho^{-|\alpha|}$ for some $\rho > 0$. Let
$$
A(x) = \sum_{\alpha} A_\alpha x^\alpha
$$
and suppose that $A_0$ is invertible. Set $\|{A_0}^{-1}\|_{op} = m$. Then, $A$ has an inverse in a neighborhood of the origin and
$$
A^{-1}(x) = \sum_{\alpha} B_\alpha x^\alpha
$$
with $\|B_\alpha\|_{op} \leq m \left ( \frac{2 M m n}{\rho} \right )^{|\alpha|} \frac{ |\alpha|!}{\alpha!}$ for any multi-index $\alpha$.
\end{lem}

\begin{proof}

In order for $A(x)A^{-1}(x) = \Id$, we must have that 
$$
A_0 B_0 = \Id
$$
and 
$$
\sum_{\alpha + \beta = \gamma} A_\alpha B_\beta = 0
$$
for all $\gamma \neq 0$. If $|\gamma| = 1$, then,
$$
A_\gamma B_0 + A_0 B_\gamma = 0 \iff B_\gamma = - B_0( A_\gamma B_0)
$$
and
$$
\|B_\gamma\|_{op} \leq m (M\rho^{-1}m) \leq m \frac{2Mmn}{\rho}.
$$

We proceed by induction on $|\gamma|$. We see that
$$
B_\gamma = - B_0  \sum_{\alpha < \gamma} A_{\gamma - \alpha} B_\alpha .
$$
Then, 
\begin{align*}
\| B_\gamma \|_{op} & \leq m  \sum_{\alpha < \gamma} \| A_{\gamma - \alpha}\|_{op} \|B_\alpha\|_{op} \\
    & \leq m  \sum_{\alpha < \gamma} M \frac{ |\gamma - \alpha|!}{(\gamma - \alpha)!} \frac{1}{\rho^{|\gamma - \alpha|}} m \left ( \frac{2 M mn}{\rho} \right )^{|\alpha|} \frac{ |\alpha|!}{\alpha!} \\
    & \leq m \left ( \frac{2M mn}{\rho} \right )^{|\gamma|} \frac{ |\gamma|!}{\gamma!} \sum_{\alpha < \gamma}  \binom{|\gamma|}{|\alpha|}^{-1} \binom{\gamma}{\alpha}   \frac{1}{ (2n)^{|\gamma - \alpha|}} \\
    & \leq m \left ( \frac{2M mn}{\rho} \right )^{|\gamma|} \frac{ |\gamma|!}{\gamma!} \sum_{\alpha < \gamma}  \frac{1}{ (2n)^{|\gamma - \alpha|}} \\
    & \leq m \left ( \frac{2M mn}{\rho} \right )^{|\gamma|} \frac{ |\gamma|!}{\gamma!} \sum_{\mu = 1}^{|\gamma|}  \frac{1}{ 2^\mu} \\
    & \leq m \left ( \frac{2M mn}{\rho} \right )^{|\gamma|} \frac{ |\gamma|!}{\gamma!}.
\end{align*}

\end{proof}

\begin{prop} \label{prop: inverse}
Let $A(x) = \sum_{\alpha} A_\alpha x^\alpha$ be an analytic matrix-valued function defined on a neighborhood of the origin in $\mathbb{C}^n$. Let $M$ and $\rho$ be positive constants such that $A \in C_{op,M,\rho}(0)$. Suppose that $A_0$ is invertible and set $m = \|{A_0}^{-1}\|_{op}$. Then, $\|A^{-1}\|_{op} \leq 2\|{A_0}^{-1}\|_{op}$ on the polydisk of radius $\frac{\rho}{4n^2Mm}$ centered at the origin.
\end{prop}

\begin{proof}
By our assumption, $\|A_\alpha\|_{op} \leq M \frac{|\alpha|!}{\alpha!} \rho^{-|\alpha|}$ for any index $\alpha$. By Lemma \ref{lem: inverse},
$$
A^{-1} = \sum_{\alpha} B_\alpha x^\alpha.
$$
with 
$$
\|B_\alpha\|_{op} \leq m \left ( \frac{2 M m n}{\rho} \right )^{|\alpha|} \frac{ |\alpha|!}{\alpha!}
$$
for all multi-indices $\alpha$. So, the identity $\frac{1}{1 - (x_1 + \dotsb + x_n)} = \sum_{\alpha} \frac{|\alpha|!}{\alpha!}x^\alpha$ implies that $\|A^{-1}\|_{op} \leq 2m$ on the polydisk of radius $\frac{\rho}{4n^2Mm}$.
\end{proof}

\begin{thm} \label{thm: bounds}
If $U$ is a sufficiently small coordinate unit disk, then $a_k^{(N)} : U \times U \times U \to M_{r_k \times r_k}(\mathbb{C})$ is a negligible amplitude modulo $O(k^{-N})$. Furthermore, there exists a constant $C_N$ independent of $k$ such that $\|b_{k,m}(x,z)\|_{op(E^k_x,E^k_x)} \leq C_N$ for all $m \leq N$, $k \in \mathbb{N}$, and $x,z \in U_N$.
\end{thm}

\begin{proof}
By construction,
\begin{align*}
& e^{\mathfrak{s}^k(\theta(x-y))} \mathfrak{s}^k(\tau(x,y,z)) a^{(N)}_k(x,y,z) \,dz \wedge dy \\
    = \hspace{2pt} & d(e^{\mathfrak{s}^k(\theta(x-y))} A^{(N)}_k(x,y,z)\, dy) -  e^{\mathfrak{s}^k(\theta(x-y))} \frac{1}{k^N} \partial_3 A_{k,N}(x,y,z) \,dz \wedge dy
\end{align*}
for all $x,y,z \in U_N$. So, it suffices to show that $\|A_{k,m}(x,y,z)\|_{op(E^k_x,E^k_x)}$, $\|b_{k,m}(x,z)\|_{op(E^k_x,E^k_x)}$, and \\ $\left \| \partial_3 A_{k,N}(x,y,z) \right \|_{op(E^k_x,E^k_x)}$ are $O(1)$ for all $m \leq N$ and $x,y,z$ in a sufficiently small coordinate unit disk.

Set $\tilde{\tau}(x,y,z) = e^{\phi(x)/2} \tau(x,y,z) e^{-\phi(x)/2}$. After replacing $U$ with a smaller coordinate unit disk centered at $x_0$, we can ensure that $\|\tilde{\tau}(x,y,z)\|_{op} = O(1)$ for all $x,y,z \in U$. Then, Proposition \ref{prop: k}, implies that $\| \mathfrak{s}^k (\tilde{\tau}(x,y,z)) \|_{op} =O(k)$ for all $x,y,z \in U$. We observe that 
$$
\mathfrak{s}^k(\tilde{\tau}(x,x,\overline{x})) = e^{\mathfrak{s}^k\phi(x)/2}\mathfrak{s}^k \left ( \tilde{F}(x,\overline{x})^t \right ) e^{-\mathfrak{s}^k\phi(x)/2}
$$
is a positive definite Hermitian matrix because $(E,h)$ is Griffiths positive. So, Proposition \ref{prop: k} also implies that the smallest eigenvalue of $\mathfrak{s}^k(\tilde{\tau}(x,x,\overline{x}))$ is $k$ times the smallest eigenvalue of $\tilde{\tau}(x,x,\overline{x})$. Since the eigenvalues of a matrix vary continuously with respect to the coefficients of the matrix and the largest eigenvalue of $\mathfrak{s}^k(\tilde{\tau}(x,x,\overline{x}))^{-1}$ is the reciprocal of the smallest eigenvalue of $\mathfrak{s}^k(\tilde{\tau}(x,x,\overline{x}))$, $\|\mathfrak{s}^k(\tilde{\tau}(x,x,\overline{x}))^{-1}\|_{op} = O(1/k)$ for all $x \in U$. As a result, Proposition \ref{prop: inverse} implies that we can replace $U$ with a smaller coordinate unit disk so that $\|\mathfrak{s}^k (\tilde{\tau}(x,y,z))\|_{op} = O(k)$ and $\| \mathfrak{s}^k (\tilde{\tau}(x,y,z))^{-1}\|_{op} =O(1/k)$ for all $x,y,z \in U$. Moreover, Proposition \ref{prop: inverse} implies that we can choose a smaller coordinate unit disk such that the estimates for $\|\mathfrak{s}^k (\tilde{\tau}(x,y,z))\|_{op}$ and $\| \mathfrak{s}^k (\tilde{\tau}(x,y,z))^{-1}\|_{op}$ hold for all $k$. 

Fix $x \in U$ with $|x| < 1/4$. For each $m \leq N$, set
\begin{gather*}
\tilde{A}_{k,m}(y,z) =   e^{\mathfrak{s}^k(\phi(x))/2} A_{k,m}(x,y,z) e^{-\mathfrak{s}^k(\phi(x))/2} \text{ and } \\  
\tilde{a}_{k,m}(y,z) = e^{\mathfrak{s}^k(\phi(x))/2} a_{k,m}(x,y,z) e^{-\mathfrak{s}^k(\phi(x))/2}.
\end{gather*}
By construction,
\begin{align*}
e^{\mathfrak{s}^k(\phi(x))/2} b_{k,0}(x,z) e^{-\mathfrak{s}^k(\phi(x))/2} & =  \frac{1}{k} \mathfrak{s}^k(\tilde{\tau}(x,x,z)) \tilde{g}(x,z)^{-1}, \\
\tilde{a}_{k,0}(y,z) & = \tilde{g}(y,z)\tilde{g}(x,z)^{-1} \mathfrak{s}^k(\tilde{\tau}(x,y,z))^{-1} \mathfrak{s}^k(\tilde{\tau}(x,x,z)) - \Id, \hspace{10pt} \text{ and } \\
(x-y)\tilde{A}_{k,0}(y,z) & = \tilde{a}_{k,0}(y,z).
\end{align*}
By what we have shown above, there exists a constant $C > 0$ independent of $k$ and $x$ such that 
$$
\| b_{k,0}(x,z) \|_{op(E^k_x,E^k_x)}, \|\tilde{a}_{k,0}(y,z)\|_{op} \leq C
$$
for all $y,z \in U$. Then, for any indices $\alpha,\beta \geq 0$,
\begin{align*}
\|\partial_1^\alpha \partial_2^\beta \tilde{a}_{k,0}(x,\overline{x})\|_{op} & = \left \|\frac{\alpha!\beta!}{(2\pi \sqrt{-1})^2} \int_{|\zeta-x|=1/2} \int_{|\xi-\overline{x}|=1/2} \frac{\tilde{a}_{k,0}(\zeta,\xi)}{(\zeta -x)^{\alpha + 1} (\xi-\overline{x})^{\beta+1}} \,d\zeta\,d\xi \right \|_{op} \\
    & \text{ by Cauchy's integral formula} \\
    & \leq \frac{\alpha! \beta!}{4\pi}  \left ( \int_{|\zeta-x|=1/2} \int_{|\xi-\overline{x}|=1/2} \left \| \frac{\tilde{a}_{k,0}(\zeta,\xi)}{(\zeta -y)^{\alpha + 1} (\xi-z)^{\beta+1}}\right \|_{op}^2  \,d\zeta\,d\xi \right)^{1/2} \\
    & \text{ by Lemma \ref{lem: Operator}} \\
    & \leq C \alpha!\beta! 2^{\alpha+\beta}.
\end{align*}
Then, for any $\rho < 1/4$ and $y,z \in U$ with $|y-x|,|z-\overline{x}| < \rho$,
\begin{align*}
\|\tilde{A}_{k,0}(y,z)\|_{op} & = \left \|\sum_{\alpha,\beta \geq 0} \frac{-1}{(\alpha+1)!\beta!}\partial_1^{\alpha+1} \partial_2^\beta \tilde{a}_{k,0}(x,\overline{x}) (y-x)^\alpha (z-\overline{x})^\beta \right \|_{op} \\
    & \leq \sum_{\alpha,\beta \geq 0} \frac{1}{(\alpha+1)!\beta!}\|\partial_1^{\alpha+1} \partial_2^\beta \tilde{a}_{k,0}(x,\overline{x})\|_{op} \rho^{\alpha+\beta} \\
    &\leq C \sum_{\alpha,\beta \geq 0} \frac{1}{\alpha+1} (2\rho)^{\alpha+\beta} \\
    & \leq C' (1-4\rho)^{-1}.
\end{align*}
Thus, by replacing $U$ with a smaller coordinate disk if necessary, we can find a constant $C_0 > 0$ independent of $k$ such that
$$
\|a_{k,0}(x,y,z)\|_{op(E^k_x, E^k_x)} ,
\|A_{k,0}(x,y,z)\|_{op(E^k_x, E^k_x)},
\|b_{k,0}(x,z)\|_{op(E^k_x, E^k_x)} \leq C_0
$$
for all $x,y,z \in U$. Moreover, for all $x, y,z \in U$ with $|x| < 1/4$, $|y-x|, |z-\overline{x}| < 1/2$,
\begin{align*}
\|\partial_2 \tilde{A}_{k,0}(y,z)\|_{op} & = \left \|\frac{1}{(2\pi \sqrt{-1})^2} \int_{|\zeta-x|=3/4} \int_{|\xi-\overline{x}|=3/4} \frac{\tilde{A}_{k,0}(\zeta,\xi)}{(\zeta -y) (\xi-z)^2} \,d\zeta\,d\xi \right \|_{op} \\
    & \text{ by Cauchy's integral formula} \\
    & \leq \frac{3}{8\pi}  \left ( \int_{|\zeta-x|=3/4} \int_{|\xi-\overline{x}|=3/4} \left \| \frac{\tilde{A}_{k,0}(\zeta,\xi)}{(\zeta -y) (\xi-z)^{2}}\right \|_{op}^2  \,d\zeta\,d\xi \right)^{1/2} \\
    & \text{ by Lemma \ref{lem: Operator}} \\
    & \leq 36 C_0.
\end{align*}
Therefore, by replacing $U$ with an even smaller coordinate unit disk and $C_0$ with a larger constant, we may assume that
$$
\left \| \partial_3 A_{k,0}(x,y,z) \right \|_{op(E^k_x,E^k_x)} \leq C_0
$$
for all $x,y,z \in U$. Note that the smaller coordinate disk we choose does not depend on $k$.

We proceed by induction. Let $m < N$ and suppose that there exists a constant $C_m > 0$ such that 
$$
\|A_{k,m}(x,y,z)\|_{op(E^k_x,E^k_x)}, \left \| \partial_3 A_{k,m}(x,y,z) \right \|_{op(E^k_x,E^k_x)} \leq C_m
$$
for all $x,y,z \in U$. Fix $x \in U$ with $|x| < 1/4$. By construction,
\begin{align*}
e^{\mathfrak{s}^k(\phi(x))/2} b_{k,m+1}(x,z) e^{-\mathfrak{s}^k(\phi(x))/2} & =  \partial_2 \tilde{A}_{k,m}(x,z) \tilde{g}(x,z)^{-1}, \\
\tilde{a}_{k,m+1}(y,z) & = k \tilde{g}(y,z)\tilde{g}(x,z)^{-1} \mathfrak{s}^k(\tilde{\tau}(x,y,z))^{-1} \partial_2 \tilde{A}_{k,m}(x,z), \hspace{10pt} \text{ and } \\
(x-y)\tilde{A}_{k,m+1}(y,z) & = \tilde{a}_{k,m+1}(y,z) - k \mathfrak{s}^k(\tilde{\tau}(x,y,z))^{-1} \partial_2 \tilde{A}_{k,m}(y,z).
\end{align*}
As a result, there exists a constant $C > 0$ independent of $k$ and $x$ such that
$$
\| b_{k,m+1}(x,z)\|_{op(E^k_x,E^k_x)}, \|\tilde{a}_{k,m+1}(y,z)\|_{op} \leq C
$$
for all $y,z \in U$ with $|y-x|,|z-\overline{x}| \leq 1/2$.

Set
$$
f_k(y,z) = \tilde{a}_{k,m+1}(y,z) - k \mathfrak{s}^k(\tilde{\tau}(x,y,z))^{-1} \partial_2 \tilde{A}_{k,m}(y,z).
$$
By replacing $C$ with a larger constant and $U$ with a smaller unit disk if necessary, we may assume that $\|f_k(y,z)\|_{op} \leq C$
for all $y,z \in U$. We use arguments similar to the ones above. For any indices $\alpha,\beta \geq 0$,
\begin{align*}
\|\partial_1^\alpha \partial_2^\beta f_k (x,\overline{x})\|_{op}  \leq \alpha!\beta! C 2^{\alpha+\beta}.
\end{align*}
Then, for any $\rho < 1/4$ and $y,z \in U$ with $|y-x|,|z-\overline{x}| < \rho$,
\begin{align*}
\|\tilde{A}_{k,m+1}(y,z)\|_{op}  \leq C' (1-4\rho)^{-1}.
\end{align*}
After replacing $U$ with a smaller coordinate unit disk if necessary, we can find a constant $C_{m+1} > 0$ independent of $k$ such that
$$
\|a_{k,m+1}(x,y,z)\|_{op(E^k_x, E^k_x)} ,
\|A_{k,m+1}(x,y,z)\|_{op(E^k_x, E^k_x)},
\|b_{k,m+1}(x,z)\|_{op(E^k_x, E^k_x)} \leq C_{m+1}
$$
for all $x,y,z \in U$. This implies that, for all $x, y,z \in U$ with $|x| < 1/4$, $|y-x|, |z-\overline{x}|\leq 1/2$,
\begin{align*}
\|\partial_2 \tilde{A}_{k,m+1}(y,z)\|_{op}  \leq 36 C_{m+1}.
\end{align*}
Therefore, by replacing $U$ with an even smaller coordinate unit disk and $C_{m+1}$ with a larger constant, we may assume that
$$
\left \| \partial_3 A_{k,m+1}(x,y,z) \right \|_{op(E^k_x,E^k_x)} \leq C_{m+1}
$$
for all $x,y,z \in U$. We again note that the smaller coordinate disk we choose does not depend on $k$.

\end{proof}

We define the local section $K_k^{(N)}: U \times U \to E^k \boxtimes \overline{E^k}$ by
\begin{equation*} 
K_k^{(N)}(y,x) = \frac{1}{2\pi} \overline{ e^{\mathfrak{s}^k\psi(x,\overline{y}))}b_k^{(N)}(x,\overline{y})}.
\end{equation*}
Then, \eqref{eqn: b_k id}, Theorem \ref{thm: local reprod mod}, and Theorem \ref{thm: bounds} yield the following corollary.

\begin{cor} \label{cor: local reprod mod}
Let $x_0 \in X$ and $N \in \mathbb{N}$. Then, there exists a coordinate unit disk $U$ centered at $x_0$ such that $K_k^{(N)}$ is a reproducing kernel modulo $O(k^{-N})$ for $H_{U,k}$. More precisely, 
$$
\left \|u(x) - \left (\chi u, K_k^{(N)}(\cdot,x) \right )_{U,k}  \right \|_{\Sym^kh(x)} = O ( {k^{-N}}  ) {(u,u)_{U,k}}^{1/2}
$$
for all $u \in H_{U,k}$, $x \in U$ with $|x|<1/4$, and $k \in \mathbb{N}$.
\end{cor}

\section{Global Asymptotic Expansion}

We have constructed local reproducing kernels that are holomorphic in the first variable and anti-holomorphic in the second variable. In this section, we will show that the local reproducing kernels approximate the global Bergman kernel. This was shown in \cite{BBS} in the case when $(E,h)$ is a positive line bundle. 

Let $(\cdot, \cdot): H^0(X,\Sym^kE) \times H^0(X,\Sym^kE) \to \mathbb{C}$ denote the inner product
$$
(\cdot, \cdot) = \int_X \langle \cdot, \cdot \rangle_{\Sym^kh} \omega.
$$
We set $d_k = \dim(H^0 (X, \Sym^k E))$. Choose an orthonormal basis $\{s_i\}_{i=1}^{d_k}$ of $H^0(X,\Sym^kE)$. The Bergman kernel $K_k$ is the section of the vector bundle $\Sym^k E \boxtimes \Sym^k \overline{E}$ given by
$$
K_k(x,y) = \sum_{i=1}^{d_k} s_i(x) \otimes \overline{s_i}(y)
$$
and satisfies the property that, for any element $s \in H^0(X,\Sym^kE)$ and $x \in X$,
$$
s(x) = (s, K_k(\cdot,x)).
$$
The Bergman function $B_k$ is the section of $(\Sym^kE)^* \boxtimes \Sym^kE$ obtained from $\overline{K_k}$ by using the isomorphism $\Sym^k \overline{E} \cong (\Sym^k E)^*$ induced by the metric $\Sym^k h$.

Let $\{e_1, \dotsc, e_{r_k}\}$ and $\{f_1, \dotsc, f_{r_k}\}$ be local frames of $\Sym^kE$ near $x$ and $y$, respectively. We can view $K_k(x,y)$ as an $r_k \times r_k$ matrix, the $i,j$-th entry of which corresponds to the coefficient of $e_i \otimes \overline{f_j}$, and $B_k(x,y)$ as a matrix acting on row vectors. Then, we see that 
$$
B_k(x,y) =  e^{-\mathfrak{s}^k(\phi(x))} \overline{K_k(x,y)} .
$$
In the orthonormal frames given by the change of basis matrix $e^{-\mathfrak{s}^k(\phi(x))/2}$ and $e^{\mathfrak{s}^k(\phi(y))/2}$, the Bergman function can be written as
\begin{equation} \label{eq : AA^*}
B_k(x,y) = e^{-\mathfrak{s}^k(\phi(x))/2} \overline{K_k(x,y)} e^{-\mathfrak{s}^k(\phi(y))/2} = \overline{A(x)}A(y)^t,
\end{equation}
where $A(z)$ is an $r_k \times d_k$ matrix whose $i$-th column is $s_i(z)$ written with respect to the chosen orthonormal frame. It follows from the Singular Value Decomposition Theorem that $\|B_k(x,x)\|_{op} = \|A(x)^t\|_{op}^2$. Moreover,
$$
\|A(x)^t\|_{op}^2 = \sup_{(v_1, \dotsc, v_{d_k}) \in \mathbb{C}^{d_k}\setminus \{0\}} \frac{ \left \| \sum_i v_is_i(x) \right \|_{\Sym^kh(x)}^2}{\sum_i |v_i|^2} = \sup_{s \in H^0(X,\Sym^kE)\setminus \{0\}} \frac{ \|s(x)\|_{\Sym^kh(x)}^2}{(s,s)}.
$$
As a result,
\begin{equation} \label{eq: extremal char}
\|B_k(x,x)\|_{op} = \sup_{s \in H^0(X,\Sym^kE) \setminus \{0\}} \frac{\|s(x)\|_{\Sym^kh(x)}^2}{(s,s)}.
\end{equation}
This is the so-called extremal characterization of the Bergman function.

We will need the following definition and lemma from \cite{LL}.

\begin{defn}
Let $U$ be a coordinate neighborhood of some point $x_0 \in X$. Let $E'$ be a holomorphic vector bundle over $X$ equipped with a real-analytic Hermitian metric $h'$. We say that a frame $\{e_1, \dotsc, e_{r'}\}$ of $E'$ on $U$ is a \emph{K-frame centered at $x_0$} if
$$
h'_{\mu\nu}(x_0) = \delta_{\mu\nu} \hspace{10pt} \text{ and } \hspace{10pt} \frac{\partial^I h'_{\mu\nu}}{\partial y^I} (x_0) = 0
$$
for any $I \geq 0$.
\end{defn}

\begin{lem} \label{lem: K-frame}
Let $E'$ be a holomorphic vector bundle equipped with a Hermitian metric $h'$. For all sufficiently small coordinate unit disk $U$ centered at $x_0$, there exists a family of $K$-frames $\{e_1(x), \dotsc, e_{r'}(x)\}$ indexed by $U$ such that the frames vary smoothly with respect to $x \in U$. Moreover, for each $x \in U$, we can write $h' = e^{-\phi_x}$ with respect to the frame $\{e_1(x), \dotsc, e_{r'}(x)\}$ and, for any $y \in U$, $\|\phi_x(y)\|_{op} = O(|y-x|^2)$ uniformly for all $x \in U$.
\end{lem}

\begin{proof}
Choose any normal holomorphic frame at $x_0 \in X$. Then, there exists a coordinate unit disk $U$ centered at $x_0$ small enough so that we can write $h' = e^{-\phi}$ on $U$ with respect to the chosen frame. We consider the change of frame given by the matrix $e^{\psi(y,\overline{x})^t} e^{-\phi(x)^t/2}$. For any fixed $x \in U$, the change of frame is holomorphic in $y$ and Lemma \ref{lem: diastatic} implies that, after replacing $U$ with a smaller coordinate unit ball if necessary, the resulting frame is a $K$-frame centered at $x \in U$.

In fact, with respect to the $K$-frame centered at $x$, we can write $h' = e^{D(x,y)}$ and 
$$
D(x,y) = \sum_{\alpha,\beta > 0} D_{\alpha\overline{\beta}}(x) (y-x)^\alpha (\overline{y} - \overline{x})^\beta
$$
for all $y \in U$. By shrinking $U$ if necessary, we may assume that $\| D_{\alpha \overline{\beta}}(x) \|_{op} \leq M (\alpha + \beta)! 5^{-\alpha -\beta}$ for all $x \in U$ and for all $\alpha,\beta \geq 0$. As a result, 
\begin{equation*}
\begin{split}
\|D(x,y)\|_{op} & \leq \sum_{\alpha,\beta > 0} \|D_{\alpha\overline{\beta}}(x)\|_{op} |y-x|^{\alpha+\beta} \\
    & = |y-x|^2 \sum_{\alpha,\beta > 0} \|D_{\alpha\overline{\beta}}(x)\|_{op} |y-x|^{\alpha+\beta-2} \\
    & \leq |y-x|^2 \frac{\frac{4}{5}M}{5 - 2|y-x|} \\
    & = O(|y-x|^2)
\end{split}
\end{equation*}
uniformly for all $x \in U$.
\end{proof}

Let $U$ be a coondinate unit disk centered at $x_0 \in X$. For any $x \in U$ and $a > 0$ sufficiently small, we denote by $\overline{B}_{a}(x)$ the subset of $U$ that is mapped onto the polydisk of radius $a$ centered at $x$ under the coordinate system. We will need the following variant of Lemma $4.1$ in \cite {RT}.

\begin{lem} \label{lem: standard estimate}
Let $X$ be a K\"ahler manifold of dimension $n$. Let $x_0 \in X$ be given and let $U$ denote a coordinate polydisk of radius $1$ centered at $x_0$. Then, there exists a constant $C$ such that 
$$
|f(x)| \leq C(a\|\overline{\partial}f\|_{C^0(\overline{B}_a(x))} + a^{-n}\|f\|_{L^2(\overline{B}_a(x))})
$$
for any smooth function $f$, $x \in U$ with $|x| < 1/2$, and sufficiently small $a > 0$.
\end{lem}

\begin{prop} \label{prop: standard estimate}
Let $X$ be a compact Riemann surface and $E$ a holomorphic vector bundle with a Hermitian metric $h$. Then, there exists a constant $C$ such that, for any section $s \in C^\infty(\Sym^kE)$ and $x \in X$, 
$$
\|s(x)\|_{\Sym^kh(x)} \leq C \left ( k^{r-(3/2)} \|\overline{\partial} s\|_{C^0(X)} + k^{r/2} (s,s)^{1/2} \right ).
$$
Additionally, there exists a constant $C'$ such that, for any $s \in H^0(X,\Sym^kE)$ and $x \in X$, 
$$
\|s(x)\|_{\Sym^kh(x)} \leq C' k^{1/2}(s,s)^{1/2}
$$
\end{prop}
\begin{proof}
Let $x_0 \in X$ be given. Since $X$ is compact, it suffices to prove the estimate in some neighborhood of $x_0$. By Lemma \ref{lem: K-frame} we can find a coordinate unit disk $U$ centered at $x_0$ and $K$-frames of $E$ centered at $x$ for each $x \in U$ such that $h = e^{-\phi_x(y)}$ with $\|\phi_x(y)\|_{op} = O(|y-x|^2)$ uniformly for all $x \in U$. We know that $\Sym^k h = e^{-\mathfrak{s}^k(\phi_x(y))}$. By Proposition \ref{prop: k}, there exists a constant $C > 0$ such that
\begin{equation} \label{eq : metric bounds}
\frac{1}{C} \Id \leq \Sym^k h \leq C \Id
\end{equation}
for all $x,y \in U$ with $|x| < 1/2$ and $|y-x| < 2k^{-1/2}$ and all $k \in \mathbb{N}$ sufficiently large.

Let $s$ be any smooth section of $\Sym^kE$. Let $x \in U$ with $|x| < 1/2$ be fixed. We write $s = (s^1, \dotsc, s^{r_k})$ with respect to the $K$-frame of $\Sym^kE$ centered at $x$. 
\begin{align*}
\|s(x)\|_{\Sym^k h(x)} & \leq  \left ( C \sum_{i=1}^{r_k} |s^i(x)|^2 \right )^{1/2} \\
    & \hspace{10pt} \text{ by \eqref{eq : metric bounds}} \\
    & \leq C^{1/2} \sum_{i=1}^{r_k} |s^i(x)| \\
    & \leq C' \sum_{i=1}^{r_k} \left ( k^{-1/2} \|\overline{\partial} s^i\|_{C^0 \left (\overline{B}_{k^{-1/2}}(x) \right )} + k^{1/2} \|s^i\|_{L^2 \left (\overline{B}_{k^{-1/2}}(x) \right )}  \right ) \\
    & \hspace{10pt} \text{ by Lemma \ref{lem: standard estimate}}.
\end{align*}
We observe that
\begin{align*}
k^{1/2} \sum_{i=1}^{r_k} \|s^i\|_{L^2\left (\overline{B}_{k^{-1/2}}(x) \right )} 
    & \leq C'' k^{ r/2} \left (  \sum_{i=1}^{r_k} \|s^i\|_{L^2\left (\overline{B}_{k^{-1/2}}(x) \right )}^2  \right )^{1/2} \\
    & \hspace{10pt} \text{ by H\"older's inequality and the fact that $r_k = O(k^{r-1})$} \\
    & \leq C^{1/2} C'' k^{r/2} \left (\int_{\overline{B}_{k^{-1/2}}(x)} \langle s, s \rangle_{\Sym^k h} \, \omega \right )^{1/2} \\
    & \hspace{10pt} \text{ by \eqref{eq : metric bounds}} \\
    & \leq C^{1/2} C'' k^{r/2} (s,s)^{1/2}.
\end{align*}
Additionally,
\begin{align*}
\sum_{i=1}^{r_k} k^{-1/2} \|\overline{\partial} s^i\|_{C^0 \left (\overline{B}_{k^{-1/2}}(x) \right )} & \leq \sum_{i=1}^{r_k} k^{-1/2} C^{1/2} \|\overline{\partial} s\|_{C^0 \left (\overline{B}_{k^{-1/2}}(x) \right )} \\
    & \hspace{10pt} \text{ by \eqref{eq : metric bounds}} \\
    & \leq C^{1/2} C''' k^{r-(3/2)} \|\overline{\partial}s\|_{C^0 \left (\overline{B}_{k^{-1/2}}(x) \right )} \\
    & \leq C^{1/2} C''' k^{r-(3/2)} \|\overline{\partial}s\|_{C^0(X)}
\end{align*}
This proves the first inequality.

Next, let $s$ be a holomorphic section of $\Sym^kE$. Let $x \in U$ with $|x| < 1/2$ be fixed. We write $s = (s^1, \dotsc, s^{r_k})$ with respect to the $K$-frame of $\Sym^kE$ centered at $x$. 
\begin{align*}
\|s(x)\|_{\Sym^k h(x)} & \leq  \left ( C \sum_{i=1}^{r_k} |s^i(x)|^2 \right )^{1/2}  \\
    & \hspace{10pt} \text{ by \eqref{eq : metric bounds}} \\
    & \leq C^{1/2} C' k^{1/2} \left ( \sum_{i=1}^{r_k}\int_{\overline{B}_{k^{-1/2}}(x)}  |s^i|^2 \, \omega \right )^{1/2} \\
    & \hspace{10pt} \text{ by Lemma \ref{lem: standard estimate}} \\
    & \leq CC'k^{1/2} \left (\int_{\overline{B}_{k^{-1/2}}(x)} \langle s, s \rangle_{\Sym^k h} \, \omega \right )^{1/2} \\
    & \hspace{10pt} \text{ by \eqref{eq : metric bounds}} \\
    & \leq CC'k^{1/2} (s,s)^{1/2}.
\end{align*}
This proves the second inequality.
\end{proof}

Proposition \ref{prop: standard estimate}, \eqref{eq: extremal char}, and the fact that $X$ is compact imply the following global bound.
\begin{thm} \label{thm: global bound}
There exists a constant $C > 0$ such that $\|B_k(x,x)\|_{op} \leq Ck$ for all $x \in X$.
\end{thm}

We recall the following facts from \cite{D}. 

\begin{thm} \textnormal{(\cite{D}, Chapter 8, Theorem 4.5)} \label{thm: Hormander} 
Let $(X,\omega)$ be a compact K\"ahler manifold of dimension $n$ and let $E'$ be a holomorphic vector bundle over $X$ with a Hermitian metric $h'$. Suppose that the Hermitian operator $A_{E',\omega} = [\sqrt{-1} F_{h'}, \Lambda]$ is positive definite in bidegree $(p,q)$ with $q > 0$. Then, for any $f \in C^{\infty}_{p,q}(E)$ with $\overline{\partial}f = 0$, there exists $g \in C^\infty_{p,q-1}(E)$ such that $\overline{\partial}g = f$ and
$$
(g,g) \leq \int_X \langle {A_{E',\omega}}^{-1}f,f \rangle_{h'} \frac{\omega^n}{n!}.
$$
\end{thm}

We will use the above theorem in the case when $E' = \Sym^kE \otimes T^{(1,0)}X$, $p = 1$, and $q = 1$. Note that $E' \otimes (T^*X)^{(1,0)}$ is isomorphic to $E$ as holomorphic vector bundles, because we can contract any element of $(T^*X)^{(1,0)}$ with $T^{(1,0)}X$. In fact, this is an isometry between the fibers. The curvature form of $E'$ is $\mathfrak{s}^k F_h \otimes 1 + \Id \otimes -\sqrt{-1} \Ric(\omega)$. Then, $\Sym^k h'$ is Griffiths positive for all $k$ sufficiently large.

\begin{prop} \textnormal{(\cite{D}, Chapter 7, Lemma 7.2)}
Let $E'$ be a holomorphic vector bundle over a compact Riemann surface $X$ and let $h'$ be a Hermitian metric on $E'$. Suppose that $(E',h')$ is Griffiths positive. Then, the Hermitian operator $[\sqrt{-1} F_{h'},\Lambda]$ is positive definite on $E' \otimes \bigwedge^{(1,1)}T^*X$. 
\end{prop}

Putting all these facts together, we obtain the following corollary.
\begin{cor} \label{cor: Hormander}
If $f_1 \in C^\infty_{0,1}(\Sym^kE)$ with $\overline{\partial}f_1 = 0$, then there exists $f_0 \in C^\infty(\Sym^kE)$ such that
$$
(f_0,f_0) \leq O\left ( \frac{1}{k} \right ) (f_1,f_1)
$$
for all $k$ sufficiently large.
\end{cor}

The following theorem shows that there exists an asymptotic expansion of the Bergman function with respect to the uniform norm in an open neighborhood.

\begin{thm} \label{thm: C_0}
Let $U$ be a sufficiently small coordinate unit disk in $X$. Then,
$$
\left \| B_k(x,y) - e^{-\mathfrak{s}^k(\phi(x))} \overline{K_k^{(N)}(x,y)} \right \|_{op(\Sym^k E_x,\Sym^kE_y)} = O\left ( k^{-N} \right )
$$
for all $x,y \in U$.
\end{thm}

\begin{proof}
Let $U$ be a sufficiently small coordinate unit disk such that the hypothesis of Corollary \ref{cor: local reprod mod} is satisfied. First, we estimate the difference
$$
\epsilon_1(y,x) = e^{-\mathfrak{s}^k(\phi(y))/2} \left (  \overline{K_k(y,x)} - \overline{\left ( \chi(\cdot) K_k(\cdot,x), K_k^{(N)}(\cdot,y) \right )} \right ) e^{-\mathfrak{s}^k(\phi(x))/2}.
$$
By definition, $\left ( \chi(\cdot) K_k(\cdot,x), K_k^{(N)}(\cdot,y) \right )$ is the element of $\left ( \Sym^kE \right )_y \otimes \left ( \Sym^k \overline{E} \right )_x$ obtained by pairing the $C^\infty(\Sym^k E)$ components of $\chi(\cdot) K_k(\cdot,x)$ and $K_k^{(N)}(\cdot,y)$ using $\Sym^k h$. With respect to a chosen frame, we can write 
$$
\left ( \chi(\cdot) K_k(\cdot,x), K_k^{(N)}(\cdot,y) \right ) =  \int_U \left (  K_k^{(N)}(\cdot ,y) \right)^* e^{-\mathfrak{s}^k(\phi(\cdot))^t}  \chi(\cdot) K_k(\cdot,x) \, \omega .
$$ 
By \eqref{eq : AA^*}, 
$$
e^{-\mathfrak{s}^k(\phi(y))^t/2} K_k(y,x) e^{-\mathfrak{s}^k(\phi(x))^t/2} = A(y) A(x)^*.
$$
Then, Corollary \ref{cor: local reprod mod} and the definition of $A(\cdot)$ imply that 
$$
e^{-\mathfrak{s}^k(\phi(y))^t/2}  \left ( \chi(\cdot) K_k(\cdot, x), K_k^{(N)}(\cdot, y) \right )e^{-\mathfrak{s}^k(\phi(x))^t/2} = (A(y)+V(y))A(x)^*,
$$
where the columns $v_i(y)$ of $V(y)$ satisfy $\|v_i(y)\|_{\mathbb{C}^{r_k}} = O(k^{-N})$. Thus,
\begin{align*}
 \hspace{10pt} \left \| \epsilon_1(y,x) \right \|_{op}    & = \left \|\overline{V(y)}A(x)^t \right \|_{op} \\
    & \leq \left \| \overline{V(y)} \right \|_{op} \|A(x)^t\|_{op} \\
    & = \|V(y)\|_{op} \|A(x)^t\|_{op} \\
    & = (\|B_k (x,x)\|_{op})^{1/2} O\left ( \frac{1}{k^{N}} \right ) \\
    & = O\left ( \frac{1}{k^{N-\frac{1}{2}}} \right ).
\end{align*}

Next, we estimate the difference
$$
\epsilon_2(y,x) = e^{-\mathfrak{s}^k(\phi(y))/2} \left ( \overline{\chi(y) K_k^{(N)}(y,x)} - \overline{\left ( \chi(\cdot) K_k^{(N)}(\cdot,x), K_k(\cdot,y) \right )} \right ) e^{-\mathfrak{s}^k(\phi(x))/2}
$$
Note that the columns of
$$
\chi(y) K_k^{(N)}(y,x) - \left ( \chi(\cdot) K_k^{(N)}(\cdot,x), K_k(\cdot,y) \right )
$$
are the Bergman projections of the columns of $\chi(\cdot) K_k^{(N)}(\cdot,x)$. In other words, the columns are $L^2$ minimal solutions of the equation
$$
\overline{\partial} V(\cdot) = \overline{\partial} \left ( \chi(\cdot) K_k^{(N)}(\cdot,x ) \right ).
$$
Since $K_k^{(N)}$ is holomorphic in the first variable,
$\overline{\partial} \left ( \chi(\cdot) K_k^{(N)}(\cdot,x) \right )=\overline{\partial} \chi(\cdot) K_k^{(N)}(\cdot,x)$. Let $v(\cdot)$ be a column of the matrix 
$$
e^{-\mathfrak{s}^k(\phi(\cdot))^t/2} \, \overline{\partial} \left ( \chi(\cdot) K_k^{(N)}(\cdot,x) \right )e^{-\mathfrak{s}^k(\phi(x))^t/2}.
$$
We see that
\begin{equation*}
\begin{split}
&  \|v(\cdot)\|_{op} \\
& \leq \overline{\partial} \chi(\cdot) \left \| e^{-\mathfrak{s}^k(\phi(\cdot))/2} e^{\mathfrak{s}^k\psi(x,\overline{\cdot})}b_k^{(N)}(x,\overline{\cdot})e^{-\mathfrak{s}^k(\phi(x))/2} 
 \right \|_{op} \\
    & \leq \overline{\partial} \chi(\cdot) \| e^{-\mathfrak{s}^k(\phi(\cdot))/2} e^{\mathfrak{s}^k\psi(x,\overline{\cdot})}e^{-\mathfrak{s}^k(\phi(x))/2} \|_{op} \| e^{\mathfrak{s}^k(\phi(x))/2} b_k^{(N)}(x,\overline{\cdot})e^{-\mathfrak{s}^k(\phi(x))/2} \|_{op} \\
    & = O(e^{-\delta k}) \|e^{\mathfrak{s}^k(\phi(x))/2} b^{(N)}_k(x,\overline{\cdot}) e^{-\mathfrak{s}^k(\phi(x))/2}\|_{op} \\
    & \text{ by Proposition \ref{prop: main estimate 2}} \\
    & = O(ke^{-\delta k}) \\
    & \text{ by Theorem \ref{thm: bounds}}.
\end{split}
\end{equation*}
As a result, the $L^2$-norm and the $C^0$-norm of $v(\cdot)$ are $O(e^{-\delta' k})$ for some $\delta' < \delta$. Then, Corollary \ref{cor: Hormander} implies that the columns of $\epsilon_2(\cdot,x)$ have $L^2$-norms that are $O(e^{-\delta' k})$. By Proposition \ref{prop: standard estimate}, there is a uniform bound on the columns of $\epsilon_2(\cdot,x)$ of order $O(e^{-\delta'' k})$ for some $\delta'' < \delta'$. It follows that 
$$
\|\epsilon_2(y,x)\|_{op}= \|\epsilon_2(y,x)^t\|_{op} = O(r_ke^{-\delta'' k}) = O(e^{-\delta''' k}),
$$
for some constant $\delta''' > 0$ smaller than $\delta''$.

Therefore, for all $|x| <1/4$ and $|y| < 1/4$,
\begin{align*}
& \left \| B_k(x,y) - e^{-\mathfrak{s}^k(\phi(x))} \overline{K_k^{(N)}(x,y)} \right \|_{op(\Sym^k E_x,\Sym^k E_y)} \\
    = \hspace{2pt} & \left \|e^{-\mathfrak{s}^k(\phi(x))/2} \left (  \overline{K_k(x,y)} - \overline{K_k^{(N)}(x,y)} \right) e^{-\mathfrak{s}^k(\phi(y))/2} \right \|_{op} \\
    = \hspace{2pt} & \left \|  \epsilon_1(y,x)^* - \epsilon_2(x,y)  \right \|_{op}  \\
    = \hspace{2pt} & O\left ( \frac{1}{k^{N- \frac{1}{2}}} \right )
\end{align*}
Set $N' = N+1$. Then,
\begin{equation*}
    \left \| B_k(x,y) - e^{-\mathfrak{s}^k(\phi(x))} \overline{K_k^{(N')}(x,y)}  \right \|_{op(\Sym^k E_x,\Sym^k E_y)} \\
    =  O \left (k^{-N-1/2} \right).
\end{equation*}
It follows that,
\begin{equation*}
\begin{split}
    & \left \| B_k(x,y) - e^{-\mathfrak{s}^k(\phi(x))} \overline{K_k^{(N)}(x,y)}  \right \|_{op(\Sym^k E_x,\Sym^k E_y)} \\
    & \leq \left \| B_k(x,y) - e^{-\mathfrak{s}^k(\phi(x))} \overline{K_k^{(N')}(x,y)} \right \|_{op(\Sym^k E_x,\Sym^k E_y)} \\
    & \hspace{5pt} + \frac{k^{-N}}{2\pi} \left \| e^{-\mathfrak{s}^k(\phi(x))}  e^{\mathfrak{s}^k\psi(y,\overline{x}))}b_{k,N+1}(y,\overline{x})\right \|_{op(\Sym^k E_x,\Sym^k E_y)} \\
    & = O\left( k^{-N} \right ).
\end{split}
\end{equation*}

\end{proof}

Because $K_k^{(N)}(x,y)$ and $K_k(x,y)$ are holomorphic in $x$ and anti-holomorphic in $y$, we can use Theorem \ref{thm: C_0} to obtain a $C^p$ expansion of the Bergman function. We will need the following lemma (c.f. Lemma 4.9 of \cite{RT}).

\begin{lem} \label{lem: Cauchy}
Let $U \subseteq \mathbb{C}$ be an open neighborhood of the origin. Let $\{f_k(x,y)\}_{k=1}^\infty$ be a sequence of $r_k \times r_k$ matrix valued functions on $U \times U$ such that $f_k$ is holomorphic for each $k$. Suppose that 
$$
\left \| e^{-\mathfrak{s}^k(\phi(x))/2} f_k(y,\overline{x}) e^{-\mathfrak{s}^k(\phi(y))/2} \right \|_{op} = O(k^q)
$$
uniformly on $U \times U$ for some integer $q$. Then, for any positive integer $p$ and for any differential operator $L$ in $x$ and $y$ of order $p$,
$$
\left \| e^{-\mathfrak{s}^k(\phi(x))/2} L (f_k)(y,\overline{x}) e^{-\mathfrak{s}^k(\phi(y))/2} \right \|_{op} = O(k^{q+p}).
$$
uniformly on $B_\rho(0) \times B_\rho(0)$ where $B_\rho(0)$ is a disk of radius $\rho$ centered at the origin and $\rho$ is sufficiently small. 
\end{lem}

\begin{proof}
Let $(x,y) \in U \times U$ and let $B_{k^{-1}}(x)$ denote the disk of radius $k^{-1}$ around $x$. For all $k$ large enough so that $B_{k^{-1}}(x) \subseteq U$, Cauchy's integral formula implies that
\begin{equation*}
\begin{split}
& e^{-\mathfrak{s}^k(\phi(x))/2} \deriv{f_k}{\overline{x}} (y,\overline{x}) e^{-\mathfrak{s}^k(\phi(y))/2} \\
    = & \int_{\partial B_{k^{-1}}(x)} e^{-\mathfrak{s}^k(\phi(x))/2}e^{\mathfrak{s}^k(\phi(\zeta))/2} \frac{e^{-\mathfrak{s}^k(\phi(\zeta))/2} f_k (y,\zeta) e^{-\mathfrak{s}^k(\phi(y))/2}}{2\pi \sqrt{-1} (\zeta-\overline{x})^2} \,d\zeta.
\end{split}
\end{equation*}
Choose $\rho > 0$ small enough so that $B_{2\rho}(0)$ is compactly contained in $U$ and the matrix-valued function $Z : B_{2\rho(0)} \times B_{2\rho}(0)$ defined by
\begin{equation*}
Z(x,\zeta) = Z^{(2)} \left (-\phi(x)/2, \phi(\zeta)/2\right )    
\end{equation*}
is real analytic.  So, for all $x \in B_{\rho}(0)$, $k$ sufficiently large, and $\zeta \in \partial B_{k^{-1}}(x)$, 
\begin{equation*}
Z(x,\zeta) = O(|\zeta-x|) = O(k^{-1})    
\end{equation*}
uniformly in $x \in B_\rho(0)$. Then, Proposition \ref{prop: k} implies that $\| e^{\mathfrak{s}^k(Z(x,\zeta))}\|_{op} = O(1)$ . Thus,   
\begin{align*}
& \left \| \int_{\partial B_{k^{-1}}(x)} e^{\mathfrak{s}^k(Z(x,\zeta))} \frac{e^{-\mathfrak{s}^k(\phi(\zeta))/2} f_k (y,\zeta) e^{-\mathfrak{s}^k(\phi(y))/2}}{2\pi \sqrt{-1} (\zeta-\overline{x})^2} \,d\zeta \right \|_{op} \\
    \leq \hspace{2pt} & \frac{1}{2\pi k^{1/2}} \left ( \int_{\partial B_{k^{-1}}(x)} \left \| e^{\mathfrak{s}^k(Z(x,\zeta))} \frac{e^{-\mathfrak{s}^k(\phi(\zeta))/2} f_k (y,\zeta) e^{-\mathfrak{s}^k(\phi(y))/2}}{2\pi \sqrt{-1} (\zeta-\overline{x})^2} \right \|_{op}^2 \,d\zeta \right )^{1/2} \\
    & \text{ by Lemma \ref{lem: Operator}} \\
    = \hspace{2pt} &  O(k^{q+1}).
\end{align*}
Other derivatives are treated similarly.

\end{proof}

Let $E'$ be a holomorphic vector bundle over $X$ and let $h'$ be a Hermitian metric on $E'$. For any positive integer $p$, we define the $C^p$ norm on the set of smooth sections of $\End(E)$ as follows. For any nonnegative integer $q$, let $\nabla_q$ denote the connection on $(T_\mathbb{C}^* X )^{\otimes q} \otimes \End(E')$ induced by the Chern connection on $E'$ and the Levi-Civita connection on $TX$. We see that
\begin{equation*}
(T_\mathbb{C}^* X )^{\otimes q} \otimes \End(E') \cong (T_\mathbb{C}^* X )^{\otimes q} \otimes E'^* \otimes E' \cong  \Hom \left ( (T_\mathbb{C} X)^{\otimes q} \otimes E', E' \right ).
\end{equation*}
We define the $C^p$ norm of $A \in C^\infty(X, \End(E))$ by
$$
\|A\|_{C^p,op} = \sup_{x \in X} \sum_{q=0}^p \left \| \nabla_{q-1} \dotsb \nabla_0 A (x) \right \|_{op \left ( (T_\mathbb{C} X)_{x}^{\otimes q} \otimes E'_x, E'_x \right )}.
$$

\begin{thm} \label{thm: Main 2}
Let $N \in \mathbb{N}$ and $p \in \mathbb{Z}_{\geq 0}$ be fixed. There exist smooth sections $b_{k,0}$, $\dotsc$, $b_{k,N}$ of $\End(\Sym^k E)$ such that $\|b_{k,i}\|_{C^p,op} = O(k^p)$ for all $i = 0, \dotsc, N$ and
$$
B_k(x,x) = b_{k,0}(x) k + \dotsb + b_{k,N}(x) k^{1-N} + O(k^{p-N})
$$
where the error term $O(k^{p-N})$ is bounded with respect to the norm $\| \cdot \|_{C^p,op}$. Moreover,
\begin{align*}
b_{k,0}(x) = \hspace{2pt} & \frac{\sqrt{-1}}{k} \Lambda F_{\Sym^k h}(x) \hspace{10pt} \text{ and } \\
    b_{k,1}(x) = \hspace{2pt} & -\frac{1}{2}(\Lambda (F_{\Sym^k h}) )^{-1}\Lambda (\Delta (F_{\Sym^k h})) (x) 
+ \frac{1}{2}\Scal_\omega(x)  \\
    & +  \frac{\sqrt{-1}}{2} \Lambda  \left ( (\Lambda(F_{\Sym^k h}))^{-1} \nabla^{1,0}(\Lambda(F_{\Sym^k h})) \wedge (\Lambda(F_{\Sym^k h}))^{-1} \nabla^{0,1}(\Lambda (F_{\Sym^k h}))   \right ) (x).
\end{align*}
\end{thm}

\begin{proof}
Let $x_0 \in X$ be given. By Theorem \ref{thm: C_0}, there exists a coordinate unit disk $U$ centered at $x_0$ and matrix valued functions $b_{k,0}, \dotsc, b_{k,N}$ defined on $U$ such that
$$
\left \|e^{-\mathfrak{s}^k(\phi(x))}\left ( \overline{K_k(x,y)} - \frac{1}{2\pi} e^{\mathfrak{s}^k(\psi(y,\overline{x}))}\left ( b_{k,0}(y,\overline{x})k + \dotsb + b_{k,N}(y,\overline{x}) k^{-N+1} \right) \right )  \right \|_{op (\Sym^k E_x, \Sym^k E_y)} = O(k^{-N})
$$
uniformly for all $x,y \in U$. Let $\rho$ be a smooth function that is compactly supported in $U$. We will show that
$$
\left \|\rho(\cdot) e^{-\mathfrak{s}^k(\phi(\cdot))} \left ( \overline{K_k(\cdot,\cdot)} - \frac{1}{2\pi}e^{\mathfrak{s}^k(\psi(\cdot,\overline{\cdot}))} \left (  b_{k,0}(\cdot,\overline{\cdot})k + \dotsb +  b_{k,N}(\cdot,\overline{\cdot}) k^{1-N} \right) \right )  \right \|_{C^p,op} = O(k^{p-N})
$$
and that
$$
\|\rho(\cdot) b_{k,i}(\cdot,\overline{\cdot})\|_{C^p,op} = O(k^p)
$$
for all $i= 0, \dotsc, N$.
Then, the theorem will follow from an argument using a partition of unity subordinate to a finite cover of $X$ consisting of open sets with the same properties as $U$.

We first consider the case when $p = 1$. Define $\eta : U \to M_{r \times r}(\mathbb{C})$ and $\xi : U \to M_{r\times r}(\mathbb{C})$ by
\begin{equation*}
    \eta(x) = -\sum_{n=1}^\infty \frac{1}{n!} \ad(-\phi(x))^{n-1}(\partial_1(\phi)(x)) 
\end{equation*}
and
\begin{equation*}
    \xi(x) = -\sum_{n=1}^\infty \frac{1}{n!} \ad(-\phi(x))^{n-1}(\overline{\partial}_1(\phi)(x)).
\end{equation*}
Note that $\eta$ is the connection matrix of $h = e^{-\phi}$ on $U$. Set 
\begin{equation*}
f_k(y,\overline{x}) =  \overline{K_k(x,y)} - \overline{K_k^{(N)}(x,y)}.
\end{equation*}
We observe that 
\begin{equation*}
\begin{split}
\nabla_0 \left ( \rho(x) e^{-\mathfrak{s}^k(\phi(x))} f_k(x,\overline{x} )\right )
    = & \hspace{3pt} \mathfrak{s}^k(\eta(x)) \rho(x) e^{-\mathfrak{s}^k(\phi(x))} f_k(x,\overline{x})  \,dx + e^{-\mathfrak{s}^k(\phi(x))} d (\rho f_k)(x)   \\
    & +  \mathfrak{s}^k(\xi(x)) \rho(x) e^{-\mathfrak{s}^k(\phi(x))} f_k(x,\overline{x})  \,d\overline{x} - \mathfrak{s}^k(\eta(x)) \rho(x) e^{-\mathfrak{s}^k(\phi(x))} f_k(x,\overline{x})  \,dx\\
    & + \rho(x)e^{-\mathfrak{s}^k(\phi(x))} f_k(x,\overline{x}) \mathfrak{s}^k(\eta(x)) \,dx \\
    = & \hspace{3pt} e^{-\mathfrak{s}^k(\phi(x))} d (\rho f_k)(x) +  \mathfrak{s}^k(\xi(x)) \rho(x) e^{-\mathfrak{s}^k(\phi(x))} f_k(x,\overline{x})  \,d\overline{x} \\
    & + \rho(x)e^{-\mathfrak{s}^k(\phi(x))} f_k(x,\overline{x}) \mathfrak{s}^k(\eta(x)) \,dx.
\end{split}
\end{equation*}
We will show that each of the terms on the right hand side are $O(k)$.

Consider the term $e^{-\mathfrak{s}^k(\phi(x))} d(\rho f_k)(x)$. Let $\{v_1,v_2\}$ be a basis for $(T_\mathbb{C}X)_x$ and let $\{\epsilon^1,\epsilon^2\}$ be the basis dual to $\{v_1,v_2\}$. Then, we can write 
\begin{equation*}
    e^{-\mathfrak{s}^k(\phi(x))} d(\rho f_k)(x) = M_1(x) + M_2(x),
\end{equation*}
where $M_i(x) \in \End(\Sym^k E)_x \otimes \text{span}(\{\epsilon^i\})$ for each $i=1,2$. Lemma \ref{lem: Cauchy} implies that, for each $i=1,2$, 
$$
\| M_i(x) \|_{op (E_x \otimes (T_\mathbb{C}X)_x , E_x)} = O(k^{1-N})
$$
uniformly for all $x \in U$. Then,
$$
\|e^{-\mathfrak{s}^k(\phi(x))} d(\rho f_k)(x)\|_{op (E_x \otimes (T_\mathbb{C}X)_x , E_x)} = O(k^{1-N}).
$$

The other terms also satisfy the same bound; for instance,
\begin{equation*}
\begin{split}
& \| \mathfrak{s}^k(\xi(x)) \rho(x) e^{-\mathfrak{s}^k(\phi(x))} f_k(x,\overline{x})  \,d\overline{x}  \|_{op (E_x \otimes (T_\mathbb{C}X)_x , E_x)} \\
    = & \left \| \mathfrak{s}^k(e^{\phi(x)/2} \xi(x) e^{-\phi(x)/2}) \rho(x) e^{-\mathfrak{s}^k(\phi(x))/2} f_k(x,\overline{x}) e^{-\mathfrak{s}^k(\phi(x))/2}  \,d\overline{x} \right \|_{op (\mathbb{C}^{r_k} \otimes (T_\mathbb{C}^*X)_x , \mathbb{C}^{r_k})} \\
    = & \hspace{2pt}  O(k^{1-N}) \\
    & \text{ by Lemma \ref{lem: Cauchy} and Proposition \ref{prop: k}}
\end{split}
\end{equation*}
uniformly for all $x \in U$. It follows that
\begin{equation*}
    \left \| \rho(x) e^{-\mathfrak{s}^k(\phi(x))} f_k(x,\overline{x} ) \right \|_{C^1,op} = O(k^{1-N}).
\end{equation*}

Next, set $g_k(y,\overline{x}) = e^{\mathfrak{s}^k(\psi(y,\overline{x}))} b_{k,i}(y,\overline{x})$. As above, 
\begin{align*}
\nabla_0 \left ( \rho(x) e^{-\mathfrak{s}^k(\phi(x))} g_k(x,\overline{x})  \right ) = & \hspace{3pt} e^{-\mathfrak{s}^k(\phi(x))} d (\rho g_k)(x) +  \mathfrak{s}^k(\xi(x)) \rho(x) e^{-\mathfrak{s}^k(\phi(x))} g_k(x,\overline{x})  \,d\overline{x} \\
    & + \rho(x)e^{-\mathfrak{s}^k(\phi(x))} g_k(x,\overline{x}) \mathfrak{s}^k(\eta(x)) \,dx.
\end{align*}
and
\begin{align*}
\left \|\rho(x) e^{-\mathfrak{s}^k(\phi(x))} g_k(x,\overline{x})  \right \|_{C^1,op} = O(k)
\end{align*}
by Theorem \ref{thm: bounds}, Lemma \ref{lem: Cauchy}, and Proposition \ref{prop: k}.

The case when $p > 1$ is handled in a similar manner. For instance, $\nabla_{p-1} \dotsb \nabla_{0} (\rho b_{k,i})$ can be written as a sum of terms that only involve $g_k(x,\overline{x})$, $\mathfrak{s}^k(\eta(x))$, $\mathfrak{s}^k(\xi(x))$, the Christoffel symbols for $(T_\mathbb{C} X)^*$, and their derivatives. The number of summands depends only on $p$ and it can be shown that the operator norm of each summand is at most $O(k^p)$.

Finally, the formulas for $B_{k,0}$ and $B_{k,1}$ follow from Proposition \ref{prop: computation} and the fact that $B_{k,0}$ and $B_{k,1}$ are constructed by using a partition of unity and the local functions $b_{k,0}$ and $b_{k,1}$.

\end{proof}

\newpage

\appendix

\section{Proof of Proposition \ref{prop: k}}

We fix $r > 0$ and let $V$ be a $r$-dimensional complex inner product space. If $\{e_1, \dotsc, e_r\}$ is a basis of $V$ and $\mathbf{n}=(n_1, \dotsc, n_r) \in (\mathbb{Z}_{\geq0})^r$, then we write 
$$
|\mathbf{n}| := \sum_{i=1}^r n_i,
$$
$$
e_\mathbf{n} := {e_1}^{n_1} \dotsb {e_r}^{n_r},
$$
and
$$
u_{\mathbf{n}} := \frac{e_\mathbf{n}}{\sqrt{\mathbf{n}!}}.
$$

\begin{prop} \label{prop: sym}
Let $M \in \End(V)$. Then, the following statements are true.
\begin{enumerate}
    \item Viewing $M$ and $\mathfrak{s}^k(M)$ as matrices with respect to the bases $\{e_1, \dotsc, e_r\}$ and $\{u_\mathbf{n} : |\mathbf{n}| = k\}$, we have that $\mathfrak{s}^k(M)^t = \mathfrak{s}^k(M^t)$ and $\overline{\mathfrak{s}^k(M)} = \mathfrak{s}^k(\overline{M})$.
    \item If $H$ is the matrix representation of the inner product on $V$ with respect to a basis $\{e_1, \dotsc, e_r\}$, then $\Sym^k H$ is the matrix representation of the induced inner product on $\Sym^k V$ with respect to the basis $\left \{ u_\mathbf{n} : |\mathbf{n}| = k \right \}$.
    \item We have that
    $$
    \|\mathfrak{s}^k(M) \|_{op} \leq C(r)k\|M\|_{op}
    $$
    for some constant $C(r) > 0$ that depends only on $r$.
    \item If $M$ is Hermitian and the set of eigenvalues of $M$ is $\{\lambda_1, \dotsc, \lambda_r\}$, then $\mathfrak{s}^k M$ is also Hermitian and the set of eigenvalues of $\mathfrak{s}^k(M)$ is 
    $$
    \left \{ \sum_{i=1}^r n_i \lambda_i : \sum_{i=1}^r n_i = k \right \}.
    $$
\end{enumerate}
\end{prop}
\begin{proof}

Let $\{e_1, \dotsc, e_r\}$ be a basis for $V$. Denote the coordinates on $\GL(V)$ induced from the chosen basis by $(a_{ij})_{1 \leq i,j \leq r}$. So, for any $A \in \GL(V)$,
$$
A(e_j) = \sum_{i=1}^r a_{ij}(A) e_i.
$$
Then, the induced automorphism $\Sym^k A \in \GL(\Sym^k V)$ is defined by
\begin{equation} \label{eq: symmat}
\Sym^k A({e_1}^{n_1} \dotsb {e_r}^{n_r}) = \left ( \sum_{i_1 =1}^r a_{i_11}(A) e_{i_1} \right )^{n_1} \dotsb \left ( \sum_{i_r =1}^r a_{i_rr}(A) e_{i_r} \right )^{n_r}.
\end{equation}
Let $(a_{\mathbf{n}\mathbf{m}})_{|\mathbf{n}|=|\mathbf{m}| = k}$ denote the coordinates on $\GL_{r_k}(\mathbb{C})$ induced from the basis $\{e_{\mathbf{n}} : |\mathbf{n}| = k\}$ for $\Sym^k V$. Using \eqref{eq: symmat}, we can compute the derivative $\mathfrak{s}^k = d \Sym^k \mid_{\Id}$. Indeed, for any $1 \leq i \leq r$,
$$
\mathfrak{s}^k \left ( \tanv{a_{ii}} \right ) = \sum_{|\mathbf{n}| = k} n_i \tanv{a_{\mathbf{n}\mathbf{n}}}
$$
and, for any $1 \leq i,j \leq r$ with $i \neq j$,
$$
\mathfrak{s}^k \left ( \tanv{a_{ij}} \right ) = \sum_{|\mathbf{n}| = k, n_j >0 } n_j \tanv{a_{f_{ij}(\mathbf{n})\mathbf{n}}}
$$
where $f_{ij}(\mathbf{n}) = (m_1, \dotsc, m_r)$ with $m_i = n_i+1$, $m_j = n_j-1$, and $m_\ell = n_\ell$ for $\ell \neq i,j$.

Let $(b_{\mathbf{n}\mathbf{m}})_{|\mathbf{n}|=|\mathbf{m}| = k}$ denote the coordinates on $\GL(\Sym^k V)$ induced from the basis $\{u_{\mathbf{n}} : |\mathbf{n}|=k\}$. Then, 
$$
b_{\mathbf{n}\mathbf{m}} = \frac{\sqrt{n_1! \dotsb n_r!}}{\sqrt{m_1! \dotsb m_r!}} a_{\mathbf{n}\mathbf{m}}
$$
and
$$
 \frac{\sqrt{n_1! \dotsb n_r!}}{\sqrt{m_1! \dotsb m_r!}} \tanv{b_{\textbf{n}\textbf{m}}} = \tanv{a_{\textbf{n}\textbf{m}}}.
$$
As a result,
\begin{gather*}
\mathfrak{s}^k \tanv{a_{ii}} = \sum_{|\mathbf{n}| = k} n_i \tanv{b_{\mathbf{n}\mathbf{n}}} \text{ and }  \\
\mathfrak{s}^k \tanv{a_{ij}} = \sum_{|\mathbf{n}| = k, n_j > 0} \sqrt{(n_i + 1)n_j} \tanv{b_{f_{ij}(\mathbf{n})\mathbf{n}}}.
\end{gather*}

With respect to the bases $\{e_1, \dotsc, e_r\}$ and $\{u_\mathbf{n} : |\mathbf{n}| = k\}$, the map $\mathfrak{s}^k : M_{r \times r}(\mathbb{C}) \to M_{r_k \times r_k}(\mathbb{C})$ is real. Moreover, for $i \neq j$,
\begin{equation*}
    \begin{split}
        \mathfrak{s}^k \left ( \tanv{a_{ij}} \right )^t & = \left ( \sum_{|\mathbf{n}| =   k, n_j >0 } \sqrt{(n_i+1)n_j} \tanv{b_{f_{ij}(\mathbf{n})\mathbf{n}}} \right )^t \\
        & = \sum_{|\mathbf{n}| =   k, n_j >0 } \sqrt{(n_i+1)n_j} \tanv{b_{\mathbf{n}f_{ij}(\mathbf{n})}} \\
        & = \sum_{|\mathbf{m}| =   k, m_i >0 } \sqrt{(m_j+1)m_i} \tanv{b_{f_{ji}(\mathbf{m})\mathbf{m}}} \\
        & = \mathfrak{s}^k \left ( \tanv{a_{ji}}\right ).
    \end{split}
\end{equation*}
Thus, $(1)$ holds.

Let $\{e_1', \dotsc, e_r'\}$ be an orthonormal basis of $V$ and let $T$ be the change of basis matrix from $\{e_1', \dotsc, e_r'\}$ to $\{e_1, \dotsc, e_r\}$. In particular, the matrix representation, with respect to the basis $\{e_1, \dotsc, e_r\}$, of the inner product on $V$ is
$$
H = T^t \overline{T}.
$$
Since $\{e_1', \dotsc, e_r'\}$ is an orthonormal basis, $\{u_\mathbf{n}' : |\mathbf{n}| = k\}$ is an orthonormal basis of $\Sym^k V$. Moreover, $\Sym^k (T)$ is the change of basis matrix from $\{u_\mathbf{n}' : |\mathbf{n}| = k\}$ to $\{u_\mathbf{n} : |\mathbf{n}| = k\}$. Thus, the matrix representation, with respect to the basis $\{u_\mathbf{n} : |\mathbf{n}| = k\}$, of the induced inner product on $\Sym^k V$ is
\begin{equation*}
\begin{split}
    \Sym^k(T)^t \overline{\Sym^k(T)} & = \Sym^k (T^t) \Sym^k(\overline{T}) \\
    & = \Sym^k (T^t\overline{T}) \\
    & = \Sym^k (H).
\end{split}
\end{equation*}
This proves $(2)$.

Now, let $M \in \End(V)$ be given. We use an orthonormal basis $\{e_1', \dotsc, e_r'\}$ for $V$ to write $M$ as the matrix $(m_{ij})_{1 \leq i,j \leq r}$. Since all norms on finite dimensional vector spaces are equivalent, there exists a constant $C'$ which depends on $r$ such that 
$$
\max_{1\leq i,j \leq r} |m_{ij}| \leq C' \|M\|_{op}.
$$
The formulas above imply that each entry in $\mathfrak{s}^k (M)$ is bounded from above by $C'k \|M\|_{op}$ and that any row or any column of $\mathfrak{s}^k (M)$ has at most $r(r-1) + 1$ nonzero entries. So the Cauchy-Schwarz inequality implies that, for any $v \in \Sym^k V$,
$$
\|\mathfrak{s}^k(M) (v)\|^2 \leq (r(r-1)+1)C'^2 k^2\|M\|_{op}^2 \|v\|^2.
$$
Thus,
$$
\|\mathfrak{s}^k (M)\|_{op} \leq Ck \|M\|_{op}
$$
for some $C > 0$ that depends only on $r$. This proves $(3)$.

To prove $(4)$, suppose that $M$ is Hermitian. Choose an orthonormal basis of $V$ consisting of eigenvectors $\{e_1', \dotsc, e_r'\}$ of $M$. Let $\lambda_i$ be the eigenvalue of $e_i$ for each $i$. So,
$$
M = \sum_{i=1}^r \lambda^i \tanv{a_{ii}}
$$
and
$$
\mathfrak{s}^k (M) = \sum_{i=1}^r \lambda^i \sum_{|\mathbf{n}| = k} n_i \tanv{b_{\mathbf{n}\mathbf{n}}} = \sum_{|\mathbf{n}| = k} \sum_{i=1}^r n_i\lambda^i \tanv{b_{\mathbf{n}\mathbf{n}}}.
$$

\end{proof}

\newpage

\section{Computing $b_{k,0}$ and $b_{k,1}$} 

\begin{prop} \label{prop: computation}
Let $x_0 \in X$ and $k \in \mathbb{N}$. Choose a normal holomorphic frame of $E$ at $x_0$. Let $U$ be a coordinate neighborhood of $x_0$ on which the Hermitian metric $h$ can be written in the form 
$$
h(x) = e^{-\phi(x)} = e^{-\psi(x,\overline{x})}.
$$ 
Let $F$ denote the curvature of $(\Sym^k E, \Sym^k h)$ and let $\Scal_\omega$ denote the scalar curvature of $\omega$. Then,
$$
b_{k,0}(x,\overline{x}) = \frac{  \sqrt{-1}\Lambda (F)(x) }{k}
$$
and
\begin{equation*}
\begin{split}
b_{k,1}(x,\overline{x})  = &  -\frac{1}{2}(\Lambda (F) )^{-1}\Lambda (\Delta (F)) (x) 
+ \frac{1}{2}\Scal_\omega \Id_{\Sym^k E} (x)  \\
    & +  \frac{\sqrt{-1}}{2} \Lambda  \left ( (\Lambda(F))^{-1} \nabla^{1,0}(\Lambda(F)) \wedge (\Lambda(F))^{-1} \nabla^{0,1}(\Lambda (F))   \right ) (x)
\end{split}
\end{equation*}
for all $x \in U$. Alternatively, we can write
\begin{equation*}
    b_{k,1}(x,\overline{x})  =    -\frac{1}{2} \Lambda \left (  \overline{\partial} \left ( (\Lambda (F) )^{-1} \overline{\partial}^* (F) \right ) \right ) (x) 
+ \frac{1}{2}\Scal_\omega \Id_{\Sym^k E} (x)
\end{equation*}
for all $x \in U$.
\end{prop}

\begin{proof}
Set 
$$
H(y,z) := \Sym^k h(y,z) = e^{-\mathfrak{s}^k(\psi(y,z))}.
$$
We rewrite \eqref{eq: tau2} as
\begin{equation} \label{eq: tau3}
\mathfrak{s}^k( \tau(x,y,z)) = \frac{H(x,z)H(y,z)^{-1} \partial_2 \left ( H(y,z) H(x,z)^{-1}\right )}{x-y}.
\end{equation}
Taking the Taylor expansion centered at $x$ of $H(x,z)H(y,z)^{-1}$ and $\partial_2\left ( H(y,z) H(x,z)^{-1} \right )$ with respect to the variable $y$, we see that
\begin{align*}
H(x,z) H(y,z)^{-1} = & \hspace{2pt} \Id -  \partial_1(H) H^{-1}(x,z) (y-x) -\frac{1}{2}  \partial_1^2 (H)  H^{-1} (x,z) (y-x)^2  \\
    & +  \partial_1 (H) H^{-1} \partial_1(H)H^{-1}(x,z) (y-x)^2 + \text{h.d.t.}
\end{align*}
and
\begin{align*}
\partial_2(H(y,z)H(x,z)^{-1}) = & \hspace{2pt} \partial_2(\partial_1 (H)H^{-1})(x,z)(y-x) + \frac{1}{2} \partial_2 (\partial_1^2 (H)H^{-1})(x,z)(y-x)^2  + \text{h.d.t}
\end{align*}
where $\text{h.d.t.}$ denotes a sum of higher degree terms. By \eqref{eq: tau3},
\begin{align*}
\mathfrak{s}^k(\tau(x,y,z)) = & \hspace{2pt}  - \partial_2(\partial_1 (H) H^{-1})(x,z) - \frac{1}{2} \partial_2(\partial_1^2(H) H^{-1})(x,z)(y-x)  \\
    & + \partial_1(H) H^{-1} \partial_2(\partial_1(H) H^{-1})(x,z)(y-x) + \text{h.d.t.}.
\end{align*}
Then,
\begin{align*}
\mathfrak{s}^k(\tau(x,y,z))^{-1} = & \hspace{2pt}  - (\partial_2(\partial_1 (H)H^{-1}))^{-1} (x,z) \\
    & + \frac{1}{2} (\partial_2(\partial_1 (H)H^{-1}))^{-1}\partial_2(\partial_1^2(H)H^{-1})(\partial_2(\partial_1 (H)H^{-1}))^{-1}(x,z)(y-x)  \\
    & - (\partial_2(\partial_1 (H)H^{-1}))^{-1}\partial_1(H) H^{-1} (x,z)(y-x) \\
    & + \text{h.d.t.}.
\end{align*}
Additionally, we can write 
$$
\omega(y) = \sqrt{-1} \tilde{g}(y,\overline{y}) \,dy \wedge d\overline{y}
$$
and
$$
\tilde{g}(y,z) = \tilde{g}(x,z) + \partial_1 \tilde{g}(x,z)(y-x) + \text{h.d.t.}.
$$

Set
$$
\tilde{F}(y,z) : = - \partial_2 (\partial_1 (H) H^{-1}) (y,z)
\hspace{10pt} \text{ and } \hspace{10pt}
\eta(y,z) := \partial_1 (H) H^{-1} (y,z).
$$
We have shown in \S5 that
\begin{align*}
b_{k,0}(x,z) & = \frac{1}{k} \mathfrak{s}^k(\tau(x,x,z)) \tilde{g}(x,z)^{-1}, \\
a_{k,0}(x,y,z) & = k\tilde{g}(y,z) \mathfrak{s}^k(\tau(x,y,z))^{-1} b_{k,0}(x,z) - \Id, \\
A_{k,0}(x,y,z) & = a_{k,0}(x,y,z) (x-y)^{-1}, \hspace{10pt} \text{ and } \\
b_{k,1}(x,z) & = \tilde{g}(x,z)^{-1} \deriv{A_{k,0}}{z}(x,x,z).
\end{align*}
So,
$$
b_{k,0}(x,z) = \frac{1}{k} \tilde{F}(x,z) \tilde{g}^{-1}(x,z),
$$
\begin{align*}
a_{k,0}(x,y,z) = & \hspace{2pt} \frac{1}{2} (\tilde{F})^{-1} (\partial_2(\partial_1^2(H)H^{-1}))(x,z) (y-x) + (\tilde{F})^{-1} \eta \tilde{F} (x,z)(y-x) \\
    &  +  \partial_1 (\tilde{g})\tilde{g}^{-1} (x,z)(y-x) + \text{h.d.t.},
\end{align*}
\begin{align*}
A_{k,0}(x,y,z) =  -\frac{1}{2} (\tilde{F})^{-1} (\partial_2(\partial_1^2(H)H^{-1}))(x,z) - (\tilde{F})^{-1} \eta \tilde{F} (x,z) -  \partial_1 (\tilde{g})\tilde{g}^{-1} (x,z)  + \text{h.d.t.},
\end{align*}
and
\begin{align*}
b_{k,1}(x,z) = & \hspace{2pt}   \frac{1}{2} \tilde{g}^{-1}  (\tilde{F})^{-1} \partial_2 (\tilde{F})(\tilde{F})^{-1} \partial_2(\partial_1^2(H)H^{-1}))  (x,z)\\
    & -\frac{1}{2} \tilde{g}^{-1}  (\tilde{F})^{-1}  \partial_2^2 (\partial_1^2(H)H^{-1})  (x,z) \\
    & + \tilde{g}^{-1} (\tilde{F})^{-1} \partial_2 (\tilde{F})(\tilde{F})^{-1} \eta \tilde{F} (x,z) \\
    & + \tilde{g}^{-1} \tilde{F} (x,z) - \tilde{g}^{-1} (\tilde{F})^{-1}\eta \partial_2 (\tilde{F})(x,z) \\
    & - \tilde{g}^{-1} \partial_2 ( \partial_1 (\tilde{g})\tilde{g}^{-1})(x,z) .
\end{align*}
Using the identities
\begin{equation*}
    \partial_2(\partial_1^2(H)H^{-1}) = -\partial_1 (\tilde{F}) - \tilde{F} \eta - \eta \tilde{F}
\end{equation*}
and
\begin{equation*}
    \partial_2^2 (\partial_1^2(H)H^{-1}) = -\partial_1\partial_2 (\tilde{F}) -\partial_2(\tilde{F}) \eta - \eta \partial_2(\tilde{F}) + 2\tilde{F} \tilde{F},
\end{equation*}
we can simplify the formula for $b_{k,1}(x,z)$ to
\begin{align*}
b_{k,1}(x,z) = & \hspace{2pt}   - \frac{1}{2} \tilde{g}^{-1}  (\tilde{F})^{-1} \partial_2 (\tilde{F})(\tilde{F})^{-1}  \partial_1 (\tilde{F})   (x,z)\\
    & + \frac{1}{2} \tilde{g}^{-1}  (\tilde{F})^{-1}  \partial_1\partial_2 (\tilde{F})   (x,z) \\
    & + \frac{1}{2}\tilde{g}^{-1} (\tilde{F})^{-1} \partial_2 (\tilde{F})(\tilde{F})^{-1} \eta \tilde{F} (x,z) \\
    &  - \frac{1}{2} \tilde{g}^{-1} (\tilde{F})^{-1}\eta \partial_2 (\tilde{F})(x,z) \\
    & - \tilde{g}^{-1} \partial_2 (\partial_1 (\tilde{g})\tilde{g}^{-1})(x,z).
\end{align*}

Now, let $F = F_{\Sym^k h}$ denote the curvature form of $\Sym^k h$. We observe that
\begin{align*}
\Scal_\omega (y) =  \hspace{2pt} & \sqrt{-1} \Lambda \overline{\partial} \left ( \partial (\tilde{g}) \tilde{g}^{-1} \right )(y) \\
    = \hspace{2pt} & -\tilde{g}^{-1} \partial_2 (  \partial_1 (\tilde{g}) \tilde{g}^{-1} ) (y,\overline{y}),
\end{align*}
\begin{align*}
\sqrt{-1} \Lambda (F)(y) =  -\tilde{g}^{-1} \tilde{F} (y,\overline{y}),
\end{align*}
\begin{align*}
\overline{\partial} \left ( \left ( \sqrt{-1}\Lambda (F) \right )^{-1} \right ) (y) 
    = & \hspace{2pt} - \partial_2 (\tilde{g}) \tilde{F}^{-1}(y,\overline{y}) \,d\overline{y} + \tilde{g} \tilde{F}^{-1} \partial_2 (\tilde{F}) \tilde{F}^{-1}(y,\overline{y}) \,d\overline{y}, 
\end{align*}
\begin{align*}
\overline{\partial}^* (F)(y) = \hspace{2pt} &  - \overline{*}_E (\overline{\partial}( \overline{*}_E (\tilde{F} \, dy \wedge d\overline{y}))) (y)\\
    = \hspace{2pt} &  -\overline{*}_E \left (\overline{\partial} \overline{\left ( \sqrt{-1} \tilde{g}^{-1} H^{-1} \tilde{F} H  \right ) }\right ) (y)\\
    = \hspace{2pt} & -H \partial_1 \left ( \tilde{g}^{-1} H^{-1} \tilde{F} H \right ) H^{-1} (y) \, dy \\
    = \hspace{2pt} & - \partial_1(\tilde{g}^{-1})\tilde{F}(y,\overline{y}) \,dy  + \tilde{g}^{-1} \eta\tilde{F}(y,\overline{y})\,dy \\
    & - \tilde{g}^{-1} \partial_1(\tilde{F}) (y,\overline{y}) \,dy - \tilde{g}^{-1}\tilde{F}\eta(y,\overline{y}) \,dy,
\end{align*}
\begin{align*}
\sqrt{-1} \Lambda  \left ( \overline{\partial}\left ( \left ( \sqrt{-1}\Lambda F \right )^{-1}\right )  \wedge \overline{\partial}^* (F) \right ) (y) = \hspace{2pt} 
    & -\tilde{g}^{-1} \partial_2(\tilde{g})\partial_1(\tilde{g}^{-1}) (y,\overline{y}) + \partial_1(\tilde{g}^{-1}) \tilde{F}^{-1} \partial_2 ( \tilde{F})  (y,\overline{y}) \\
    & - \partial_2(\tilde{g}^{-1}) \tilde{F}^{-1} \eta \tilde{F} (y,\overline{y}) - \tilde{g}^{-1} \tilde{F}^{-1} \partial_2(\tilde{F}) \tilde{F}^{-1} \eta \tilde{F} (y,\overline{y})\\
    & + \partial_2(\tilde{g}^{-1}) \tilde{F}^{-1} \partial_1(\tilde{F}) (y,\overline{y}) + \tilde{g}^{-1} \tilde{F}^{-1} \partial_2(\tilde{F}) \tilde{F}^{-1} \partial_1(\tilde{F}) (y,\overline{y}) \\
    & + \partial_2(\tilde{g}^{-1}) \eta (y,\overline{y}) + \tilde{g}^{-1} \tilde{F}^{-1} \partial_2(\tilde{F}) \eta (y,\overline{y}),
\end{align*}
and
\begin{align*}
\sqrt{-1} \Lambda (\Delta (F))(y) = \hspace{2pt} & \sqrt{-1}\Lambda (\overline{\partial}( \overline{\partial}^* (F))) (y) \\
    = \hspace{2pt} & + \tilde{g}^{-1}\partial_1\partial_2(\tilde{g}^{-1}) \tilde{F}(y,\overline{y}) + \tilde{g}^{-1} \partial_1(\tilde{g}^{-1}) \partial_2\tilde{F}(y,\overline{y}) \\
    & - \tilde{g}^{-1} \partial_2(\tilde{g}^{-1}) \eta \tilde{F} (y,\overline{y}) - \tilde{g}^{-2} \eta \partial_2 \tilde{F}(y,\overline{y}) \\
    & +\tilde{g}^{-1} \partial_2(\tilde{g}^{-1}) \partial_1\tilde{F}(y,\overline{y}) + \tilde{g}^{-2} \partial_1\partial_2\tilde{F}(y,\overline{y}) \\
    & + \tilde{g}^{-1} \partial_2(\tilde{g}^{-1}) \tilde{F} \eta(y,\overline{y}) + \tilde{g}^{-2} \partial_2\tilde{F}\eta(y,\overline{y}).
\end{align*}
Therefore,
$$
b_{k,0}(x,\overline{x}) = \frac{  \sqrt{-1}\Lambda (F)(x) }{k}
$$
and
$$
b_{k,1}(x,\overline{x}) = -\frac{1}{2}(\Lambda (F) )^{-1} \Lambda (\Delta (F)) (x) 
+ \frac{1}{2}\Scal_\omega(x) -  \frac{1}{2} \Lambda  \left ( \overline{\partial} \left ( (\Lambda (F))^{-1} \right ) \wedge \overline{\partial}^* (F)  \right ) (x).
$$
Note that the Bianchi identity and the Nakano identity imply that
\begin{equation*}
    \overline{\partial}^* F = \sqrt{-1} \nabla^{1,0}(\Lambda( F)).
\end{equation*}
As a result,
\begin{equation*}
\begin{split}
b_{k,1}(x,\overline{x})  = &  -\frac{1}{2}(\Lambda (F) )^{-1}\Lambda (\Delta (F)) (x) 
+ \frac{1}{2}\Scal_\omega(x)  \\
    & +  \frac{\sqrt{-1}}{2} \Lambda  \left ( (\Lambda(F))^{-1} \nabla^{0,1}(\Lambda(F)) \wedge (\Lambda(F))^{-1} \nabla^{1,0}(\Lambda (F))   \right ) (x).
\end{split}
\end{equation*}
On the other hand, the equation
\begin{equation*}
(\Lambda (F) )^{-1} \Lambda (\Delta (F)) (x) 
+ \Lambda  \left ( \overline{\partial} \left ( (\Lambda (F))^{-1} \right ) \wedge \overline{\partial}^* (F)  \right ) (x) =  \Lambda \left (  \overline{\partial} \left ( (\Lambda (F) )^{-1} \overline{\partial}^* (F) \right ) \right ) (x) 
\end{equation*}
implies that 
\begin{equation*}
    b_{k,1}(x,\overline{x})  =    -\frac{1}{2} \Lambda \left (  \overline{\partial} \left ( (\Lambda (F) )^{-1} \overline{\partial}^* (F) \right ) \right ) (x) 
+ \frac{1}{2}\Scal_\omega(x).
\end{equation*}

\end{proof}

\begin{cor} \label{cor: Riemann-Roch}
We have that
\begin{equation*}
    \dim H^0(X, \Sym^k E) = \int_X \tr (F_{\Sym^k h}) + \frac{1}{2} r_k \int_X \Scal_\omega  \omega + O(r_k k^{-1}),
\end{equation*}
where $r_k = \rk(\Sym^k E)$.
\end{cor}

\begin{proof}
As before, let $\{s_i\}_{i=1}^{d_k}$ be an orthonormal basis of $H^0(X, \Sym^k E)$. By \eqref{eq : AA^*},
\begin{equation*}
\tr (B_k(x,x)) = \tr\left ( \overline{A(x)}A(x)^t \right ),
\end{equation*}
where $A(x)$ is an $r_k \times d_k$ matrix whose $i$-th column is $s_i(z)$ written with respect to an orthonormal frame. Thus,
\begin{equation*}
    \int_X \tr (B_k(x,x)) \omega = \int_X \sum_{i=1}^{d_k} |s_i(x)|_{\Sym^k h}^2 \omega = \dim H^0(X, \Sym^k E).
\end{equation*}
On the other hand, Theorem \ref{thm: Main} implies that
\begin{equation*}
\begin{split}
    & \int_X \tr (B_k(x,x)) \omega \\
    = & \int_X \tr(F_{\Sym^k h}) + \frac{1}{2} r_k \int_X \Scal_\omega \omega  -\frac{1}{2} \int_X  \tr\left (  \overline{\partial} \left ( (\Lambda (F) )^{-1} \overline{\partial}^* (F) \right ) \right ) +  O(r_k k^{-1}).
\end{split}
\end{equation*}
The corollary now follows from the fact that
\begin{equation*}
    \int_X  \tr\left (  \overline{\partial} \left ( (\Lambda (F) )^{-1} \overline{\partial}^* (F) \right ) \right ) =  \int_X d \left (\tr \left (   (\Lambda (F) )^{-1} \overline{\partial}^* (F) \right ) \right ) = 0.
\end{equation*}
\end{proof}

\begin{cor} \label{cor: line bundle computation}
Suppose that $E$ is a positive line bundle. Set $\omega' = \sqrt{-1} F_h$. Then,
\begin{equation*}
b_{k,0}(x) = \Lambda_\omega \omega' (x)
\end{equation*}
and
\begin{equation*}
b_{k,1}(x) =  \Scal_\omega (x) - \frac{1}{2} \Lambda_\omega \omega' (x) \Scal_{\omega'} (x).
\end{equation*}
\end{cor}
\begin{proof}
We see that 
$$
b_{k,0}(x) = \frac{\sqrt{-1}}{k} \Lambda_\omega F_{\Sym_k h} = \Lambda_\omega \omega' (x).
$$
Because $E$ is a line bundle, the formula for $b_{k,1}$ reduces to
\begin{equation*}
\begin{split}
b_{k,1}(x,\overline{x}) = & - \frac{1}{2} \tilde{g}^{-1}  (\tilde{F})^{-1} \partial_2 (\tilde{F})(\tilde{F})^{-1}  \partial_1 (\tilde{F})   (x,\overline{x})\\
    & + \frac{1}{2} \tilde{g}^{-1}  (\tilde{F})^{-1}  \partial_1\partial_2 (\tilde{F})   (x,\overline{x}) \\
    & - \tilde{g}^{-1} \partial_2 (\tilde{g}^{-1} \partial_1 (\tilde{g}))(x,\overline{x}) \\
    = & \hspace{3pt} \frac{1}{2}{\tilde{g}}^{-1} \partial_2 \left ( {\tilde{F}}^{-1} \partial_1 \tilde{F}\right ) (x,\overline{x}) - \tilde{g}^{-1} \partial_2 (\tilde{g}^{-1} \partial_1 (\tilde{g}))(x,\overline{x}) \\
    = & \hspace{3pt} \Scal_\omega (x) - \frac{1}{2} \Lambda_\omega \omega' (x) \Scal_{\omega'} (x)
\end{split}
\end{equation*}

\end{proof}

Corollary \ref{cor: line bundle computation} recovers the computations in \cite{BBS}. Indeed, the results from \cite{BBS} show that
$$
b_{k,0}(x) \Lambda_{\omega'}\omega(x) = 1
$$
and
\begin{align*}
b_{k,1}(x) \Lambda_{\omega'}\omega(x)  = \Lambda_{\omega'}\omega(x) \Scal_\omega(x) - \frac{1}{2} \Scal_{\omega'}(x). 
\end{align*}

\begin{cor} \label{cor: direct sum}
Let $L_1$ and $L_2$ be line bundles with positive real analytic Hermitian metrics $h_1$ and $h_2$. For any $a,b \geq 0$, let $B_{a,b}$ denote the Bergman function associated to $(L_1^a \otimes L_2^b, h_1^a \otimes h_2^b)$ and $\omega$. Let $N$ and $p$ be any fixed nonnegative integer. Then, there exist smooth functions $b_{a,b,0}, \dotsc, b_{a,b,N}$ such that
\begin{equation*}
    B_{a,b}(x) = b_{a,b,0}(x)(a+b) + \dotsb + b_{a,b,N}(x) (a+b)^{1-N} + O \left ((a+b)^{-N} \right )
\end{equation*}
where the error term $O \left ((a+b)^{-N} \right )$ is bounded with respect to the $C^p$ norm. Furthermore,
\begin{equation*}
    b_{a,b,0}(x) = \frac{a}{a+b} \Lambda_\omega \omega_1 + \frac{b}{a+b} \Lambda_\omega \omega_2
\end{equation*}
and
\begin{equation*}
b_{a,b,1}(x) = \Scal_\omega(x) -\frac{1}{2}  \left ( a \Lambda_\omega \omega_1 + b \Lambda_\omega \omega_2 \right ) \Scal_{a\omega_1 + b\omega_2}(x). 
\end{equation*}
where $\omega_j$ is $\sqrt{-1}$ times the curvature of $h_j$ for each $j=1,2$.
\end{cor}
\begin{proof}
The formulas for $b_{a,b,i}$ follow from Corollary \ref{cor: line bundle computation} and the fact that $\Sym^k h = \bigoplus_{i=0}^k h_1^i \otimes h_2^{k-i}$. It remains to show that the asymptotic expansion holds with respect to the $C^p$ norm.

Let $U \subseteq X$ be a sufficiently small coordinate neighborhood and let $\rho \in C_c^\infty(U)$. We know that
\begin{equation*}
\left \|\rho(x) \left ( e^{-\mathfrak{s}^k(\phi(x))} \overline{K_k(x,x)} - \frac{1}{2\pi} \left ( b_{k,0}(x,\overline{x})k + \dotsb + b_{k,N+p}(x,\overline{x}) k^{-N-p+1} \right) \right) \right \|_{C^p,op} = O(k^{-N})
\end{equation*}
and that
\begin{equation*}
    \|b_{k,i}(x,\overline{y})\|_{op(\Sym^k E_x,\Sym^k E_x)} = O(1).
\end{equation*}
As $\Sym^k E$ is a direct sum of line bundles, $\End(\Sym^k E)$ is a direct sum of trivial bundles. Moreover, the formula for $b_{k,i}(x,\overline{y})$ implies that $b_{k,i}(x,\overline{y})$ is a diagonal matrix for each $i$. In particular, the operator norm of $b_{k,i}(x,\overline{y})$ is just the maximum of the absolute value of the diagonal entries. Since $b_{k,i}(x,\overline{y})$ is analytic in $x$ and $\overline{y}$, Cauchy's integral formula implies that
\begin{equation*}
    \|\rho(x) b_{k,i}(x,\overline{x})\|_{C^p,op} = O(1).
\end{equation*}
Therefore,
$$
\left \|\rho(x) \left ( e^{-\mathfrak{s}^k(\phi(x))} \overline{K_k(x,x)} - \frac{1}{2\pi} \left ( b_{k,0}(x,\overline{x})k + \dotsb + b_{k,N}(x,\overline{x}) k^{-N+1} \right) \right) \right \|_{C^p,op} = O(k^{-N}).
$$

\end{proof}

\begin{cor} \label{cor: HE}
Suppose that $h$ is a Hermitian-Einstein metric. Let $N$ and $p$ be fixed nonnegative integers. Then, there exist smooth functions $b_{1}(x), \dotsc, b_N(x)$, that do not depend on $k$, such that
\begin{equation*}
    B_k(x,x) = \left ( b_{0}(x) \Id_{\Sym^k E}\right )k + \dotsb + \left ( b_{N}(x) \Id_{\Sym^k E} \right ) k^{1-N} + O(k^{-N})
\end{equation*}
where the error term is bounded with respect to the $C^p$ norm. In particular,
\begin{equation*}
b_{0}(x) = c
\end{equation*}
where $c\Id_E = \sqrt{-1} \Lambda F_h$ and 
\begin{equation*}
b_{1}(x) = \frac{1}{2}\Scal_\omega(x).
\end{equation*}
\end{cor}
\begin{proof}

Let $U \subseteq X$ be a sufficiently small coordinate neighborhood and let $\rho \in C_c^\infty(U)$. We know that
\begin{equation*}
\left \|\rho(x) \left ( e^{-\mathfrak{s}^k(\phi(x))} \overline{K_k(x,x)} - \frac{1}{2\pi} \left ( b_{k,0}(x,\overline{x})k + \dotsb + b_{k,N+p}(x,\overline{x}) k^{-N-p+1} \right) \right) \right \|_{C^p,op} = O(k^{-N})
\end{equation*}
and that
\begin{equation*}
    \|b_{k,i}(x,\overline{y})\|_{op(\Sym^k E_x,\Sym^k E_x)} = O(1).
\end{equation*}
We will show that $b_{k,i}(x,\overline{y})$ are of the form $b_{k,i}(x,\overline{y}) = b_i(x,\overline{y}) \Id_{\Sym^k E}$, where $b_i(x,\overline{y})$ is analytic in $x$ and $\overline{y}$, for each $i$. Then, Cauchy's integral formula and the fact that 
\begin{equation*}
    \nabla_0 (b_i(x,\overline{x}) \Id_{\Sym^k E}) = d (b_i)(x,\overline{x}) \otimes \Id_{\Sym^k E}
\end{equation*}
will imply that
\begin{equation*}
    \|\rho(x)b_{k,i}(x,\overline{x})\|_{C^p,op} = O(1).
\end{equation*}

Set 
\begin{equation*}
H(x,z) := h(x,z) = e^{-\psi(x,z)}
\end{equation*}
and write
\begin{equation*}
    \omega(y) = \sqrt{-1} \tilde{g}(y,\overline{y}) dy \wedge d\overline{y}.
\end{equation*}
We know that
\begin{equation*}
\begin{split}
    \tau(x,y,z) = &  - \partial_2(\partial_1 (H) {H}^{-1})(x,z) \\
    & - \frac{1}{2} \partial_2(\partial_1^2(H) {H}^{-1})(x,z)(y-x)  \\
    & + \partial_1(H) {H}^{-1} \partial_2(\partial_1(H) {H}^{-1})(x,z)(y-x) \\
    & + \text{h.d.t.}.
\end{split}
\end{equation*}
On the other hand,
\begin{equation*}
\begin{split}
    \tilde{F}(x,z) & := -\partial_2(\partial_1 (H){H}^{-1})(x,z) \\
    & = \tilde{g}(x,z) \sqrt{-1} \Lambda F(x,z) \\
    & =  c k \tilde{g}(x,z).
\end{split}
\end{equation*}
where $F = F_{h}$ is the curvature form of $h$. In particular, $\tilde{F}(x,z)$ is a scalar function and 
\begin{equation*}
    \begin{split}
        \partial_1 \tilde{F}(x,z) & = -\partial_2(\partial_1^2 (H){H}^{-1})(x,z) + \partial_2 (\partial_1 (H){H}^{-1} \partial_1(H) {H}^{-1})  \\
        & = -\partial_2(\partial_1^2 (H){H}^{-1})(x,z) + 2 \partial_1 (H){H}^{-1} \partial_2 (\partial_1(H) {H}^{-1})  \\
\end{split}
\end{equation*}
is also a scalar function. So, we can write
\begin{equation*}
\begin{split}
    \mathfrak{s}^k(\tau(x,y,z)) =   kf_0(x,z) + kf_1(x,z) (y-x) + \text{h.d.t.}
\end{split}
\end{equation*}
where $f_0(x,z)$ and $f_1(x,z)$ are scalar functions.

We start with the case when $i = 0$. By our observations from above,
\begin{equation*}
    b_{k,0}(x,z)  = \frac{1}{k} \mathfrak{s}^k(\tau(x,x,z)) \tilde{g}(x,z)^{-1} = c  \Id.
\end{equation*}
Proceeding by induction, suppose that $i > 0$ and that \begin{equation*}
    b_{k,i-1} (x,z) = b_{i-1}(x,z)
\end{equation*}
for some analytic function $b_{i-1}(x,z) \Id_{\Sym^k E}$. If $i = 1$, then
\begin{equation*}
\begin{split}
    (x-y)A_{k,0}(x,y,z) & = k \mathfrak{s}^k(\tau(x,y,z))^{-1} \tilde{g}(y,z) b_{k,0}(x,z) - \Id_{\Sym^k E}  \\
    & = \alpha(x,z) (y-x) + \text{h.d.t.}
\end{split}
\end{equation*}
for some scalar function $\alpha(x,z)$ that does not depend on $k$. Similarly, if $i > 1$, then
\begin{equation*}
\begin{split}
    (x-y) A_{k,i-1}(x,y,z) & = k \mathfrak{s}^k(\tau(x,y,z))^{-1} \tilde{g}(y,z) b_{k,i-1}(x,z)  -  k \mathfrak{s}^k(\tau(x,x,z))^{-1} \tilde{g}(x,z) b_{k,i-1}(x,z)\\
    & = \alpha(x,z) (y-x) + \text{h.d.t.}.
\end{split}
\end{equation*}
So, $A_{k,i-1}(x,x,z) = -\alpha(x,z)$ is a scalar function. Thus,
\begin{equation*}
b_{k,i}(x,z) = \tilde{g}(x,z)^{-1} \deriv{A_{k,i-1}}{z}(x,x,z)
\end{equation*}
is a scalar function and we can write
$$
b_{k,i}(x,z) = b_i(x,z) \Id_{\Sym^k E}
$$
where $b_i(x,z)$ is a function that does not depend on $k$.

Finally, we will compute $b_1(x,\overline{x})$. Since $h$ is Hermitian-Einstein, 
\begin{equation*}
\nabla^{0,1} (\Lambda (F_{\Sym^k h})) = \overline{\partial} (\Lambda (F_{\Sym^k h})) = 0.    
\end{equation*}
Additionally,
\begin{equation*}
\begin{split}
\sqrt{-1} \Lambda (\Delta (F_{ h}))(y) =  & - \tilde{g}^{-1}\partial_1\partial_2(\tilde{g}^{-1}) \tilde{F}(y,\overline{y}) - \tilde{g}^{-1} \partial_1(\tilde{g}^{-1}) \partial_2\tilde{F}(y,\overline{y}) \\
    & + \tilde{g}^{-1} \partial_2(\tilde{g}^{-1}) \eta \tilde{F} (y,\overline{y}) + \tilde{g}^{-2} \eta \partial_2 \tilde{F}(y,\overline{y}) \\
    & -\tilde{g}^{-1} \partial_2(\tilde{g}^{-1}) \partial_1\tilde{F}(y,\overline{y}) - \tilde{g}^{-2} \partial_1\partial_2\tilde{F}(y,\overline{y}) \\
    & - \tilde{g}^{-1} \partial_2(\tilde{g}^{-1}) \tilde{F} \eta(y,\overline{y}) - \tilde{g}^{-2} \partial_2\tilde{F}\eta(y,\overline{y}) \\
    = & \hspace{3pt} \tilde{g}^{-1} \partial_1\partial_2(-\tilde{g}^{-1} \tilde{F}) (y,\overline{y}) \\
    & - \tilde{g}^{-1} \eta \partial_2(-\tilde{g}^{-1} \tilde{F}) (y,\overline{y})\\
    & +  \tilde{g}^{-1} \partial_2(-\tilde{g}^{-1} \tilde{F}) \eta (y,\overline{y}). \\
\end{split}
\end{equation*}
Since $-\tilde{g}^{-1} \tilde{F}(y,\overline{y}) = \sqrt{-1}\Lambda (F_{ h}) = c \Id$, we deduce that $\sqrt{-1} \Lambda(\Delta(F_{h})) = 0$. Then, 
\begin{equation*}
    \sqrt{-1} \Lambda(\Delta(F_{\Sym^k h})) =  \mathfrak{s}^k (\sqrt{-1} \Lambda(\Delta(F_{h}))) = 0.
\end{equation*}
Thus,
\begin{equation*}
 b_{k,1}(x,\overline{x}) = \frac{1}{2} \Scal_\omega(x) \Id_{\Sym^k E}.
\end{equation*}

\end{proof}

\end{document}